\documentclass[amstex,12pt, amssymb]{article}

\usepackage{mathtext}
\usepackage[cp1251]{inputenc}
\usepackage[T2A]{fontenc}
\usepackage[dvips]{graphicx}
\usepackage{amsmath}
\usepackage{amssymb}
\usepackage{amsxtra}
\usepackage{latexsym}
\usepackage{ifthen}

\newcounter{lemma}[section]

\newcounter{corollary}[section]

\newcounter{remark}[section]

\newcounter{theorem}[section]

\newcounter{proposition}[section]

\newcounter{example}

\numberwithin{equation}{section}

\begin{document}

\markboth{\centerline{E.~SEVOST'YANOV}} {\centerline{ON EXTENSION OF
MAPPINGS }}

\def\cc{\setcounter{equation}{0}
\setcounter{figure}{0}\setcounter{table}{0}}

\overfullrule=0pt


\author{{E.~SEVOST'YANOV}\\}

\title{
{\bf ON THE INVERSE POLETSKII INEQUALITY IN METRIC SPACES AND PRIME
ENDS}}

\date{\today}
\maketitle

\begin{abstract} We study mappings defined in the domain of a metric
space that distort the modulus of families of paths by the type of
the inverse Poletskii inequality. Under certain conditions, it is
proved that such mappings have a continuous extension to the
boundary of the domain in terms of prime ends. Under some additional
conditions, the families of such mappings are equicontinuous in the
closure of the domain with respect to the space of prime ends.
\end{abstract}


\section{Introduction} The paper is devoted to the study of mappings in metric spaces,
in particular, their extension to the boundary of the domain in
terms of prime ends (see, for example, \cite{A}, \cite{ABBS},
\cite{GRY}, \cite{KLS} and \cite{KR}). In our relatively recent
paper~\cite{Sev$_1$}, we investigated the question of the continuous
extension of mappings satisfying Poletskii's inequality in metric
spaces. The key point of the above article is the connection between
the cuts of the domain and the families of curves joining the
conjentric spheres, which, in turn, is possible due to~\cite[Theorem
10.8]{ABBS}. This article deals with mappings whose inverse
satisfies modular inequalities. This issue was partially considered
in~\cite{Sev$_1$}, see Theorems~4.1 and~4.2. However, we will
consider here an even more general case when the mapping, generally
speaking, does not have the inverse, but only satisfies a modular
estimate of a certain form (which coincides with Poletskii's
inequality for inverse mappings, if they exist). For this purpose,
consider the following definition.

\medskip
Everywhere further $(X, d, \mu)$ and $\left(X^{\,\prime},
d^{\,\prime}, \mu^{\,\prime}\right)$ are metric spaces with metrics
$d$ and $d^{\,\prime}$ and locally finite Borel measures $\mu$ and
$\mu^{\,\prime},$ correspondingly. We will assume that $\mu$ is a
Borel measure such that $0 <\mu(B)<\infty$ for all balls $B$ in $X.$
Let $y_0\in X^{\,\prime},$ $0<r_1<r_2<\infty$ and
\begin{equation}\label{eq1**}
A(y_0, r_1,r_2)=\left\{ y\,\in\,X^{\,\prime}: r_1<d^{\,\prime}(y,
y_0)<r_2\right\}\,,\end{equation}
$$B(y_0, r)=\{y\in X^{\,\prime}: d^{\,\prime}(y,
y_0)<r\}\,, S(y_0, r)=\{y\in X^{\,\prime}: d^{\,\prime}(y,
y_0)=r\}\,.$$
Given sets $E,$ $F\subset\overline{{\Bbb R}^n}$ and a domain
$D\subset X^{\,\prime}$ denote by $\Gamma(E,F,D)$ the family of all
paths $\gamma:[a,b]\rightarrow X^{\,\prime}$ such that $\gamma(a)\in
E,\gamma(b)\in\,F$ and $\gamma(t)\in D$ for $t \in [a, b].$ Given a
domain $D\subset X,$ a {\it mapping} $f:D\rightarrow X^{\,\prime}$
is an arbitrary continuous transformation $x\mapsto f(x).$ Let
$f:D\rightarrow X^{\,\prime},$ let $y_0\in f(D)$ and let
$0<r_1<r_2<d_0=\sup\limits_{y\in f(D)}d^{\,\prime}(y, y_0).$ Now, we
denote by $\Gamma_f(y_0, r_1, r_2)$ the family of all paths $\gamma$
in $D$ such that $f(\gamma)\in \Gamma(S(y_0, r_1), S(y_0, r_2),
A(y_0,r_1,r_2)).$

\medskip
Recall, for a given continuous path $\gamma:[a, b]\rightarrow X$ in
a metric space $(X, d),$ that its length is the supremum of the sums
$$
\sum\limits^{k}_{i=1} d(\gamma(t_i),\gamma(t_{i-1}))
$$
over all partitions $a=t_0\leqslant t_1\leqslant\ldots\leqslant
t_k=b$ of the interval $ [a,b].$ The path $\gamma$ is called {\it
rectifiable} if its length is finite.
\medskip

\medskip
Given a family of paths $\Gamma$ in $X$, a Borel function
$\rho:X\rightarrow[0,\infty]$ is called {\it admissible} for
$\Gamma$, abbr. $\rho\in {\rm adm}\,\Gamma$, if
\begin{equation}\label{eq13.2}
\int\limits_{\gamma}\rho\,ds \geqslant 1
\end{equation}
for all (locally rectifiable) $\gamma\in\Gamma$.
\medskip

\medskip
Given $p\geqslant 1,$ the $p$-modulus of the family $\Gamma$ is the
number
\begin{equation}\label{eq13.5}
M_p(\Gamma)=\inf\limits_{\rho\in {\rm
adm}\,\Gamma}\int\limits_X\rho^{\,p}(x)\, d\mu(x)\,.
\end{equation}
Should ${\rm adm\,}\Gamma$ be empty, we set $M_p(\Gamma)=\infty.$ A
family of paths $\Gamma_1$ in $X$ is said to be {\it minorized} by a
family of paths $\Gamma_2$ in $X,$ abbr. $\Gamma_1>\Gamma_2,$ if,
for every path $\gamma_1\in\Gamma_1$, there is a path
$\gamma_2\in\Gamma_1$ such that $\gamma_2$ is a restriction of
$\gamma_1.$ In this case,
\begin{equation}\label{eq32*A}
\Gamma_1
> \Gamma_2 \quad \Rightarrow \quad M_p(\Gamma_1)\le M_p(\Gamma_2)
\end{equation} (see~\cite[Theorem~1]{Fu}).
Let $Q:X^{\,\prime}\rightarrow [0, \infty]$ be a Lebesgue measurable
function such that $Q(y)\equiv 0$ for $y\in X^{\,\prime}\setminus
D.$ Assume that $D$ and $X^{\,\prime}$ have a finite Hausdorff
dimensions $\alpha$ and $\alpha^{\,\prime}\geqslant 1,$
respectively. We will say that {\it $f$ satisfies the inverse
Poletsky inequality} at a point $y_0\in f(D),$ if the relation
\begin{equation}\label{eq2*A}
M_{\alpha}(\Gamma_f(y_0, r_1, r_2))\leqslant \int\limits_{f(D)\cap
A(y_0,r_1,r_2)} Q(y)\cdot \eta^{{\alpha}^{\prime}}(d^{\,\prime}(y,
y_0))\, d\mu^{\,\prime}(y)
\end{equation}
holds for any Lebesgue measurable function $\eta:
(r_1,r_2)\rightarrow [0,\infty ]$ such that
\begin{equation}\label{eqA2}
\int\limits_{r_1}^{r_2}\eta(r)\, dr\geqslant 1\,.
\end{equation}
Note that the inequalities~(\ref{eq2*A}) are well known in the
theory of quasiregular mappings and hold for $Q=N(f, D)\cdot K, $
where $N(f, D)$ is the maximal multiplicity of $f$ in $D,$ and
$K\geqslant 1$ is some constant that may be calculated in the
following way: $K={\rm ess \sup}\, K_O(x, f),$ $K_O(x, f)=\Vert
f^{\,\prime}(x)\Vert^n/|J(x, f)|$ for $J(x, f)\ne 0;$ $K_O(x, f)=1$
for $f^{\,\prime}(x)=0,$ and $K_O(x, f)=\infty$ for
$f^{\,\prime}(x)\ne 0,$ where $J(x, f)=0$ (see, e.g.,
\cite[Theorem~3.2]{MRV} or \cite[Theorem~6.7.II]{Ri}).

\medskip
Let $X$ and $X^{\,\prime}$ be metric spaces. A mapping
$f:X\rightarrow X^{\,\prime}$ is {\it discrete} if $f^{\,-1}(y)$ is
discrete for all $y\in X^{\,\prime}$ and $f$ is {\it open} if $f$
maps open sets onto open sets. A mapping $f:G\rightarrow
X^{\,\prime}$ is {\it closed} in $G\subset X$ if $f(A)$ is closed in
$f(G)$ whenever $A$ closed in $G.$ From now on we assume that the
space $X^{\,\prime}$ is complete and supports a
$\alpha^{\,\prime}$-Poincar\'{e} inequality, and that the measure
$µ^{\,\prime}$ is doubling (see \cite{ABBS}). In this case, a space
$X^{\,\prime}$ is locally connected (see \cite[Section~2]{ABBS}),
and proper (see \cite[Proposition~3.1]{BB}).

\medskip
The definition and construction of prime ends used in this paper
corresponds to publications~\cite{ABBS} and~\cite{Sev$_1$}. We call
a bounded connected set $E \varsubsetneq \Omega $ an {\it
acceptable} set if $\overline{E}\cap\partial\Omega\ne\varnothing.$
By discussion in \cite{ABBS}, we know that boundedness and
connectedness of an acceptable set $E$ implies that $\overline{E}$
is compact and connected. Furthermore, $E$ is infinite, as otherwise
we would have $\overline{E}=E\subset \Omega.$ Therefore,
$\overline{E}$ is a continuum. Recall that a {\it continuum} is a
connected compact set containing at least two points.

\medskip
In what follows, given $A, B\subset X,$
$${\rm dist}\,(A, B)=\sup\limits_{x\in A, y\in B}d(x, y)\,.$$
We call a sequence $\left\{E_k\right\}_{k=1}^{\infty}$ of acceptable
sets a chain if it satisfies the following conditions:

1. $E_{k+1}\subset E_k$ for all $k=1, 2,\ldots,$

2. ${\rm dist}\,(\Omega\cap \partial E_{k+1}, \Omega\cap \partial
E_k)>0$ for all $k=1, 2,\ldots,$

3. The impression
$\bigcap\limits_{k=1}^{\infty}\overline{E_k}\subset \partial
\Omega.$

Here we have used a notation ${\rm dist}(A, B):=\inf\limits_{x\in A,
y\in B}d(x, y),$ where $d$ is a metric in a given metric space.
We say that a chain $\left\{E_k\right\}_{k=1}^{\infty}$ divides the
chain $\left\{F_k\right\}_{k=1}^{\infty}$ if for each $k$ there
exists $l_k$ such that $E_{l_k}\subset F_k.$ Two chains are
equivalent if they divide each other. A collection of all mutually
equivalent chains is called an {\it end} and denoted $[E_k],$ where
$\left\{E_k\right\}_{k=1}^{\infty}$ is any of the chains in the
equivalence class. The impression of $[E_k],$ denoted $I[E_k],$ is
defined as the impression of any representative chain. The
collection of all ends is called the {\it end boundary} and is
denoted $\partial_E\Omega.$ We say that an end $[E_k]$ is a {\it
prime end} if it is not divisible by any other end. The collection
of all prime ends is called the {\it prime end boundary} and is
denoted $E_{\Omega}.$

\medskip
In what follows, we set $\overline{\Omega}_P:=\Omega\cup
E_{\Omega}.$ We say that $\Omega$ is {\it finitely connected} at a
point $x_0\in\partial \Omega$ if for every $r>0$ there is an open
set $G$ (open in $X$) such that $x_0\in G \subset B(x_0, r)$ and
$G\cap\Omega$ has only finitely many components. If $\Omega$ is
finitely connected at every boundary point, then it is called {\it
finitely connected} at the boundary. The following results have been
proved in~\cite{ABBS}.

\medskip
\begin{proposition}\label{pr3}
{\sl Assume that $\Omega$ is finitely connected at the boundary.
Then all prime ends have singleton impressions, and every $x\in
\partial\Omega$ is the impression of a prime end and is accessible (see~\cite[Theorem~10.8]{ABBS}). }
\end{proposition}

\medskip
We say that a sequence of points $\{x_n\}_{n=1}^{\infty}$ in
$\Omega$ converges to the end $[E_k],$ and write $x_n\rightarrow
[E_k]$ as $n\rightarrow\infty,$ if for all $k$ there exists $n_k$
such that $x_n\in E_k$ whenever $n\geqslant n_k.$ The following most
important statement is true.

\medskip
\begin{proposition}\label{pr2}
{\sl Assume that $\Omega$ is finitely connected at the boundary.
Then the prime end closure $\overline{\Omega}_P$ is metrizable with
some metric $m_P:\overline{\Omega}_P\times
\overline{\Omega}_P\rightarrow {\Bbb R}$ such that the topology on
$\overline{\Omega}_P$ given by this metric is equivalent to the
topology given by the sequential convergence discussed above
(see~\cite[Corollary~10.9]{ABBS}).}
\end{proposition}

\medskip
Let us give the following definition (see~\cite[section~13.3]{MRSY},
cf.~\cite[Definition~17.5(4)]{Va} and \cite[Definition~2.8]{Na}).
Let $(X,d,\mu)$ be metric space with finite Hausdorff dimension
$\alpha\geqslant 1.$ We say that the boundary of $D$ is {\it weakly
flat} at a point $x_0\in
\partial D$ if, for every number $P > 0$ and every neighborhood $U$
of the point $x_0,$ there is a neighborhood $V\subset U$ such that
$M_{\alpha}(\Gamma(E, F, D))\geqslant  P$ for all continua $E$ and
$F$ in $D$ intersecting $\partial U$ and $\partial V.$ We say that
the boundary $\partial D$ is weakly flat if the corresponding
property holds at every point of the boundary.

\medskip
One of the main statements of the article is the following theorem.

\medskip
\begin{theorem}\label{th1A}
{\sl Let $D$ and $D^{\,\prime}$ be domains with finite Hausdorff
dimensions $\alpha$ and $\alpha^{\,\prime}\geqslant 2$ in spaces
$(X,d,\mu)$ and $(X^{\,\prime},d^{\,\prime}, \mu^{\,\prime}),$
respectively. Assume that $X$ is locally connected, $\overline{D}$
is compact, $X^{\,\prime}$ is complete and supports
$\alpha^{\,\prime}$-Poincar\'{e} inequality, and that the measures
$\mu$ and $µ^{\,\prime}$ are doubling. Let $D^{\,\prime}\subset
X^{\,\prime}$ be a bounded domain which is finitely connected at the
boundary, and let $Q:X^{\,\prime}\rightarrow (0, \infty)$ be
integrable function in $D^{\,\prime},$ $Q(y)\equiv 0$ for $y\in
X^{\,\prime}\setminus D^{\,\prime}.$ Suppose that $f:D\rightarrow
D^{\,\prime},$ $D^{\,\prime}=f(D),$ is an open discrete and closed
mapping satisfying the relation~(\ref{eq2*A}) at any point $y_0\in
\partial D^{\,\prime},$ moreover, suppose that $D$ has a weakly flat
boundary and $\overline{D}$ is compact in $X.$ Then $f$ has a
continuous extension
$\overline{f}:\overline{D}\rightarrow\overline{D^{\,\prime}}_P$ such
that $\overline{f}(\overline{D})=\overline{D^{\,\prime}}_P.$ }
\end{theorem}

\medskip
Given $\delta>0, M>0$ domains $D\subset X, D^{\,\prime}\subset
X^{\,\prime},$ and a continuum $A\subset D^{\,\prime}$ denote by
${\frak S}_{\delta, A, M}(D, D^{\,\prime})$ a family of all open
discrete and closed mappings $f$ of $D$ onto $D^{\,\prime}$ such
that the condition~(\ref{eq2*A}) holds for some $Q=Q_f$ for any
$y_0\in D^{\,\prime}$ and such that $d(f^{\,-1}(A),
\partial D)\geqslant~\delta$ and $\Vert Q_f\Vert_{L^1(D^{\,\prime})}\leqslant M.$

\medskip
Now let us talk about the equicontinuity of families of mappings in
the closure of a domain. Statements like the ones below have been
established in various situations in~\cite[Theorem~1.2]{SSD} and
\cite[Theorem~2]{Sev$_2$}. To such a great degree of generality,
this statement is proved for the first time.

\medskip
\begin{theorem}\label{th2}
{\sl Let $D$ and $D^{\,\prime}$ be domains with finite Hausdorff
dimensions $\alpha$ and $\alpha^{\,\prime}\geqslant 2$ in spaces
$(X,d,\mu)$ and $(X^{\,\prime},d^{\,\prime}, \mu^{\,\prime}),$
respectively. Assume that $X$ is locally connected, $\overline{D}$
and $\overline{D^{\,\prime}}$ are compact sets, $D$ has a weakly
flat boundary, $D$ is weakly flat as a metric space, $X^{\,\prime}$
is complete and supports $\alpha^{\,\prime}$-Poincar\'{e}
inequality, and that the measures $\mu$ and $µ^{\,\prime}$ are
doubling. Let $D^{\,\prime}\subset X^{\,\prime}$ be a regular domain
which is finitely connected at the boundary. Then any $f\in{\frak
S}_{\delta, A, M}(D, D^{\,\prime})$ has a continuous extension
$\overline{f}:\overline{D}\rightarrow \overline{D^{\,\prime}}_P,$
wherein $\overline{f}(\overline{D})=\overline{D^{\,\prime}}_P$ and,
in addition, the family ${\frak S}_{\delta, A, M }(\overline{D},
\overline{D^{\,\prime}})$ of all extended mappings
$\overline{f}:\overline{D}\rightarrow \overline{D^{\,\prime}}_P$ is
equicontinuous in $\overline{D}.$ }
\end{theorem}

\section{Continuous extension of mappings to the boundary }

\medskip
Before proving the main result, we give some more definitions, and
also prove one important statement.

\medskip Given a domain $D\subset X,$ the {\it cluster set} of
$f:D\rightarrow Y$ at $b\in \partial D$ is the set $C(f, b)$ of all
points $z\in Y$ for which there exists a sequence
$\{b_k\}_{k=1}^{\infty}$ in $D$ such that $b_k\rightarrow b$ and
$f(b_k)\rightarrow z$ as $k\rightarrow\infty.$ For a non-empty set
$E\subset \partial D$ let $C(f, E)=\cup C(f, b),$ where $b$ ranges
over set $E.$

\medskip Let $D\subset X,$ $f:D \rightarrow X^{\,\prime}$ be a
discrete open mapping, $\beta: [a,\,b)\rightarrow X^{\,\prime}$ be a
path, and $x\in\,f^{-1}\left(\beta(a)\right).$ A path $\alpha:
[a,\,c)\rightarrow D$ is called a {\it maximal $f$-lifting} of
$\beta$ starting at $x,$ if $(1)\quad \alpha(a)=x\,;$ $(2)\quad
f\circ\alpha=\beta|_{[a,\,c)};$ $(3)$\quad for
$c<c^{\prime}\leqslant b,$ there is no path $\alpha^{\prime}:
[a,\,c^{\prime})\rightarrow D$ such that
$\alpha=\alpha^{\prime}|_{[a,\,c)}$ and $f\circ
\alpha^{\,\prime}=\beta|_{[a,\,c^{\prime})}.$ If $X$ and
$X^{\,\prime}$ are locally compact, $X$ is locally connected, and
$f:D \rightarrow X^{\,\prime}$ is discrete and open, then there is a
maximal $f$-lifting of $\beta$ starting at $x,$
see~\cite[Lemma~2.1]{SM}. We also prove an even more general
statement (for the space ${\Bbb R}^n$ see, for example,
\cite[Theorem~3.7]{Vu}).

\medskip
\begin{lemma}\label{lem9}
{\sl Let $X$ and $X^{\,\prime}$ be metric spaces, let $X$ be locally
connected, let $X^{\,\prime}$ be locally compact, let $D$ be a
domain in $X,$ and let $f:D \rightarrow X^{\,\prime}$ be a discrete
open and closed mapping of $D$ onto $D^{\,\prime}\subset
X^{\,\prime}.$ Assume that $\overline{D}$ is compact. If $\beta:
[a,\,b)\rightarrow X^{\,\prime}$ be a path such that
$x\in\,f^{\,-1}(\beta(a)),$ then there is a whole $f$-lifting of
$\beta$ starting at $x,$ in other words, there is a path $\alpha:
[a,\,b)\rightarrow X$ such that $f(\alpha(t))=\beta(t)$ for any
$t\in [a,\,b).$ Moreover, if $\beta(t)$ has a limit
$\lim\limits_{t\rightarrow b-0}\beta(t):=B_0\in D^{\,\prime},$ then
$\alpha$ has a continuous extension to $b$ and $f(\alpha(b))=B_0.$ }
\end{lemma}

\medskip
\begin{proof}
Since $f|_D\rightarrow X^{\,\prime}$ is a mapping of the locally
compact space $D$ to $X^{\,\prime},$ in addition, $X$ is locally
connected and $X^{\,\prime}$ is locally compact, the existence of a
maximal $f$-lifting $\alpha: [a,\,b)\rightarrow X$ follows from
Lemma~\cite[Lemma~2.1]{SM}. Let us prove that this lifting is whole,
for which we use the general scheme from~\cite[Proof of
Lemma~2.1]{Sev$_1$}, cf.~\cite[Lemma~1]{SeSkSv}. Suppose the
opposite, namely that $c\ne b.$ Note that $\alpha$ can not tend to
the boundary of $D$ as $t\rightarrow c-0,$ since $C(f,
\partial D)\subset \partial D^{\,\prime}$ by Proposition~2.1 in \cite{Sev$_1$}. Then $C(\alpha, c)\subset
D.$

\medskip
Consider
$$G=\left\{x\in X:\, x=\lim\limits_{k\rightarrow\,\infty}
\alpha(t_k)
 \right\}\,,\quad t_k\,\in\,[a,\,c)\,,\quad
 \lim\limits_{k\rightarrow\infty}t_k=c\,.$$
Letting to subsequences, we may restrict us by monotone sequences
$t_k.$ For $x\in G,$ by continuity of $f,$
$f\left(\alpha(t_k)\right)\rightarrow\,f(x)$ as
$k\rightarrow\infty,$ where $t_k\in[a,\,c),\,t_k\rightarrow c$ as
$k\rightarrow \infty.$ However,
$f\left(\alpha(t_k)\right)=\beta(t_k)\rightarrow\beta(c)$ as
$k\rightarrow\infty.$ Thus, $f$ is a constant on $G.$ On the other
hand, $\overline{\alpha}$ is a compact set, because
$\overline{\alpha}$ is a closed subset of the compact space
$\overline{D}$ (see \cite[Theorem~2.II.4, $\S\,41$]{Ku}). Now, by
Cantor condition on the compact $\overline{\alpha},$ by monotonicity
of $\alpha([t_k,\,c)),$
$$G\,=\,\bigcap\limits_{k\,=\,1}^{\infty}\,\overline{\alpha\left(\left[t_k,\,c\right)\right)}
\ne\varnothing\,,
$$
%
see~\cite[1.II.4, $\S\,41$]{Ku}. Now, by~\cite[Theorem~5.II.5,
$\S\,47$]{Ku}, $\overline{\alpha}$ is connected. By discreteness of
$f,$ $G$ is a single-point set, and $\alpha\colon
[a,\,c)\rightarrow\,D$ extends to a closed path $\alpha\colon
[a,\,c]\rightarrow D$ and $f(\alpha(c))=\beta(c).$ Then,
by~\cite[Lemma~2.1]{SM} there is a new $f$-lifting
$\alpha^{\,\prime}: [c, c^{\,\prime})\rightarrow D$ of $\beta$
starting at $\alpha(c).$ Uniting liftings $\alpha$ and
$\alpha^{\,\prime}$ we obtain a new $f$-lifting $\alpha:[a,
c^{\,\prime})\rightarrow D$ starting at $x$ that contradicts the
maximality of $\alpha.$ The contradiction obtained above proves that
$c=b.$

\medskip
Let there exits $\lim\limits_{t\rightarrow b-0}\beta(t):=B_0\in
D^{\,\prime}.$ Arguing similarly to what was proved above, we obtain
that the set $G$ consists of one point $p_0\in D$ and, therefore,
the path $\alpha$ extends to the closed path $\alpha:[a,
c]\rightarrow D.$ The equality $f(\alpha(b))=B_0$ follows from the
continuity of the mapping $f.$~$\Box$
\end{proof}

\medskip
{\it Proof of Theorem~\ref{th1A}.} Put $x_0\in\partial D.$ It is
necessary to show the possibility of continuous extension of the
mapping $f$ to $x_0.$ Assume that the conclusion about the
continuous extension of the mapping $f$ to the point $x_0$ is not
correct. Then any prime end $P_0 \in E_{D^{\,\prime}}$ is not a
limit of $f$ at the point $x_0.$ It follows that, there is some
sequence $x_k\in D,$ $k=1,2,\ldots,$ $x_k\rightarrow x_0 $ as
$k\rightarrow \infty,$ and a number $\varepsilon_0>0 $ such that
$m_P(f(x_k), P_0)\geqslant \varepsilon_0$ for any $k\in {\Bbb N},$
where $m_P$ is defined in Proposition~\ref{pr2}.
By~\cite[Theorem~10.10]{ABBS}, $\left(\overline{D^{\,\prime}}_P,
m_P\right)$ is a compact metric space. Thus, we may assume that
$f(x_k)$ converges to some $P_1\ne P_0,$
$P_1\in\overline{D^{\,\prime}}_P$ as $k\rightarrow\infty.$ Since $f$
has no a limit at $x_0$ by the assumption, there is some sequence
$y_k,$ $y_k\rightarrow x_0$ as $k\rightarrow\infty,$ such that
$m_P(f(y_k), P_1)\geqslant \varepsilon_1$ for any $k\in {\Bbb N}$
and some $\varepsilon_1>0.$ Since the space
$(\overline{D^{\,\prime}}_P, m_P)$ is compact, we may assume that
$f(y_k)\rightarrow P_2$ as $k\rightarrow \infty,$ $P_1\ne P_2,$
$P_2\in \overline{D^{\,\prime}}_P.$ Since $f$ is closed, $f$ is
boundary preserving, as well (see
e.g.~\cite[Proposition~2.1]{Sev$_1$}. Thus, $P_1, P_2\in
E_{D^{\,\prime}}.$

\medskip
The rest of the proof is close enough to the proof of Lemma~4.1
in~\cite{Sev$_1$}. Let $P_1=[E_k],$ $k=1,2,\ldots,$ and $P_2=[G_l],$
$l=1,2,\ldots, .$ By Remark 4.5 in \cite{ABBS} we may consider that
the sets $E_k$ and $G_l$ are open. By the assumption $X^{\,\prime}$
is complete and supports a $\alpha^{\,\prime}$-Poincar\'{e}
inequality, and that the measure $µ^{\,\prime}$ is doubling (see
\cite{ABBS}). In this case, a space $X^{\,\prime}$ is quasiconvex
(see \cite[Theorem~17.1]{Ch}) and, consequently, $X^{\,\prime}$ is
locally path connected. By Mazurkiewicz–Moore–Menger theorem,
$X^{\,\prime}$ is locally arcwise connected, see \cite[Theorem~1,
Ch.~6, $\S$ 50, item II]{Ku}. Since $E_k$ and $G_l$ are domains,
they are path connected for any $k, l\in {\Bbb N}$
(see~\cite[Theorem~2.I.50, Ch.6]{Ku}).

\medskip
Let us show that there exists $k_0\in {\Bbb N}$ such that
\begin{equation}\label{eq12}
E_k\cap G_k=\varnothing\quad \forall\,\,k\geqslant k_0\,.
\end{equation}
Suppose the contrary, i.e., suppose that for every $l=1,2,\ldots$
there exists an increasing sequence $k_l,$ $l=1,2,\ldots,$ such that
$x_{k_l}\in E_{k_l}\cap G_{k_l},$ $l=1,2,\ldots .$ Now
$x_{k_l}\rightarrow P_1$ and $x_{k_l}\rightarrow P_2,$
$l\rightarrow\infty.$ Let $m_P$ be the metric on
$\overline{D^{\,\prime}}_P$ defined in Proposition \ref{pr2}. By
triangle inequality,
$$m_P(P_1, P_2)\leqslant m_P(P_1, x_{k_l})+m_P(x_{k_l}, P_2)
\rightarrow 0,\qquad l\rightarrow\infty\,,$$
that contradicts to Proposition~\ref{pr2}. Thus, (\ref{eq12}) holds,
as required.

\medskip
Denote $y_0:=I([E_k])$ (see Proposition \ref{pr3}). Arguing as in
the proof of Lemma~2.1 in~\cite{Sev$_1$}, we may show that, for
every $r>0$ there exists $N\in {\Bbb N}$ such that
\begin{equation}\label{eq10} E_k\subset
B(y_0, r)\cap D^{\,\prime}\quad \forall\,\, k\geqslant N\,.
\end{equation}
Since $D^{\,\prime}$ is connected and $E_{k_0+1}\ne D^{\,\prime},$
we obtain that $\partial E_{k_0+1}\cap D^{\,\prime}\ne\varnothing$
(see \cite[Ch.~5, $\S\,$46, item I]{Ku}). Set
$r_0:=d^{\,\prime}(y_0,
\partial E_{k_0+1}\cap D^{\,\prime}).$ Since $\overline{E_{k_0}}$ is compact, $r_0>0.$
By (\ref{eq10}), there is $m_0\in {\Bbb N},$ $m_0>k_0+1,$ such that
\begin{equation}\label{eq10A} E_k\subset
B(y_0, r_0/2)\cap D^{\,\prime}\quad \forall\,\, k\geqslant m_0\,.
\end{equation}

Set $D_0:=E_{m_0+1},$ $D_*:=G_{m_0+1}.$ Let us to show that
\begin{equation}\label{eq11}
\Gamma(D_0, D_*, D^{\,\prime})>\Gamma(S(y_0, r_0/2), S(y_0, r_0),
A(y_0, r_0/2, r_0))\,,
\end{equation}
where $A(y_0, r_1, r_2)$ is defined in (\ref{eq1**}).
Assume that $\gamma\in \Gamma(D_0, D_*, D^{\,\prime}),$ $\gamma:[0,
1]\rightarrow D^{\,\prime}.$ Set
$$|\gamma|:=\{x\in D^{\,\prime}: \exists\,t\in[0, 1]:
\gamma(t)=x\}\,.$$
By (\ref{eq12}), $|\gamma|\cap E_{k_0+1}\ne\varnothing\ne
|\gamma|\cap (D^{\,\prime}\setminus E_{k_0+1}).$ Thus,
\begin{equation}\label{eq13}
|\gamma|\cap \partial E_{k_0+1}\ne\varnothing
\end{equation} (see \cite[Theorem 1, $\S\,$46, item I]{Ku}).
Moreover, observe that
\begin{equation}\label{eq14}
\gamma(1)\not\in \partial E_{k_0+1}\,.
\end{equation}
Suppose the contrary, i.e., that $\gamma(1)\in \partial E_{k_0+1}.$
By definition of prime end, $\partial E_{k_0+1}\cap
D^{\,\prime}\subset \overline{E_{k_0}}.$ Since ${\rm
dist}\,(D^{\,\prime}\cap
\partial E_{k+1}, D^{\,\prime}\cap \partial E_k)>0$ for all $k=1, 2,\ldots,$ we
obtain that $\partial E_{k_0+1}\cap D^{\,\prime}\subset E_{k_0}.$
Now, we have that $\gamma(1)\in E_{k_0}$ and, simultaneously,
$\gamma(1)\in G_{m_0+1}\subset G_{k_0}.$ The last relations
contradict with (\ref{eq12}). Thus, (\ref{eq14}) holds, as required.

\medskip
By~(\ref{eq10A}), we obtain that $|\gamma|\cap B(y_0,
r_0/2)\ne\varnothing.$ We prove that $|\gamma|\cap
(D^{\,\prime}\setminus B(y_0, r_0/2))\ne\varnothing.$ In fact, if it
is not true, then $\gamma(t)\in B(y_0, r_0/2)$ for every $t\in [0,
1].$ However, by (\ref{eq13}) we obtain that $(\partial
E_{k_0+1}\cap D^{\,\prime})\cap B(y_0, r_0/2)\ne \varnothing,$ that
contradicts to the definition of $r_0.$ Thus, $|\gamma|\cap
(D\setminus B(y_0, r_0/2))\ne\varnothing,$ as required. Now, by
\cite[Theorem 1, $\S\,$46, item I]{Ku}, there exists $t_1\in (0, 1]$
with $\gamma(t_1)\in S(y_0, r_0/2).$ We may consider that
$t_1=\max\{t\in [0, 1]: \gamma(t)\in S(y_0, r_0/2)\}.$ We prove that
$t_1\ne 1.$ Suppose the contrary, i.e., suppose that $t_1=1.$ Now,
we obtain that $\gamma(t)\in B(y_0, r_0/2)$ for every $t\in [0, 1).$
On the other hand, by (\ref{eq13}) and (\ref{eq14}), we obtain that
$\partial E_{k_0+1}\cap B(y_0, r_0/2)\ne \varnothing,$ which
contradicts to the definition of $r_0.$ Thus, $t_1\ne 1,$ as
required. Set $\gamma_1:=\gamma|_{[t_1, 1]}.$

\medskip
By the definition, $|\gamma_1|\cap B(y_0, r_0)\ne\varnothing.$ We
prove that $|\gamma_1|\cap (D^{\,\prime}\setminus B(y_0,
r_0))\ne\varnothing.$ In fact, assume the contrary, i.e., assume
that $\gamma_1(t)\in B(y_0, r_0)$ for every $t\in [t_1, 1].$ Since
$\gamma(t)\in B(y_0, r_0/2)$ for $t<t_1,$ by (\ref{eq13}) we obtain
that $|\gamma_1|\cap
\partial E_{k_0+1}\ne\varnothing.$ Consequently, $B(y_0, r_0)\cap (\partial E_{k_0+1}\cap D^{\,\prime})\ne\varnothing,$
that contradicts to the definition of $r_0.$ Thus, $|\gamma_1|\cap
(D^{\,\prime}\setminus B(y_0, r_0))\ne\varnothing,$ as required.
Now, by \cite[Theorem 1, $\S\,$46, item I]{Ku}, there exists $t_2\in
(t_1, 1]$ with $\gamma(t_2)\in S(y_0, r).$ We may consider that
$t_2=\min\{t\in [t_1, 1]: \gamma(t)\in S(y_0, r_0)\}.$ We put
$\gamma_2:=\gamma|_{[t_1, t_2]}.$ Observe that $\gamma>\gamma_2$ and
$\gamma_2\in\Gamma(S(y_0, r_0/2), S(y_0, r_0), A(y_0, r_0/2, r_0)).$
Thus, (\ref{eq11}) has been proved.

\medskip
Since $x_k\rightarrow P_1$ as $k\rightarrow\infty$ and
$y_l\rightarrow P_2$ as $l\rightarrow\infty,$ there is $M_0\in {\Bbb
N}$ such that $x_k\in D_0$ and $y_l\in D_*$ for any $k\geqslant M_0$
and $l\geqslant M_0.$ Without loss of generality we may assume that
$x_k\in D_0$ and $y_l\in D_*$ for any $k, l\in {\Bbb N}.$ Let
$\gamma_k$ be a path joining $f(x_1)$ and $f(x_k)$ in $D_0,$ and
$\gamma^{\,\prime}_k$ be a path joining $f(y_1)$ and $f(y_k)$ in
$g_1.$ Since by the assumption $X^{\,\prime}$ is complete and the
measure $µ^{\,\prime}$ is doubling, $X^{\,\prime}$ is proper (see
\cite[Proposition~3.1]{BB}); in particular, $X^{\,\prime}$  is
locally compact. Now, by Lemma~\ref{lem9} paths $\gamma_k$ and
$\gamma^{\,\prime}_k$ have whole $f$-liftings $\alpha_k$ and
$\beta_k$ in $D$ starting at points $x_k$ and $y_k,$ respectively
(see Figure~\ref{fig1A}).
\begin{figure}[h]
\centerline{\includegraphics[scale=0.55]{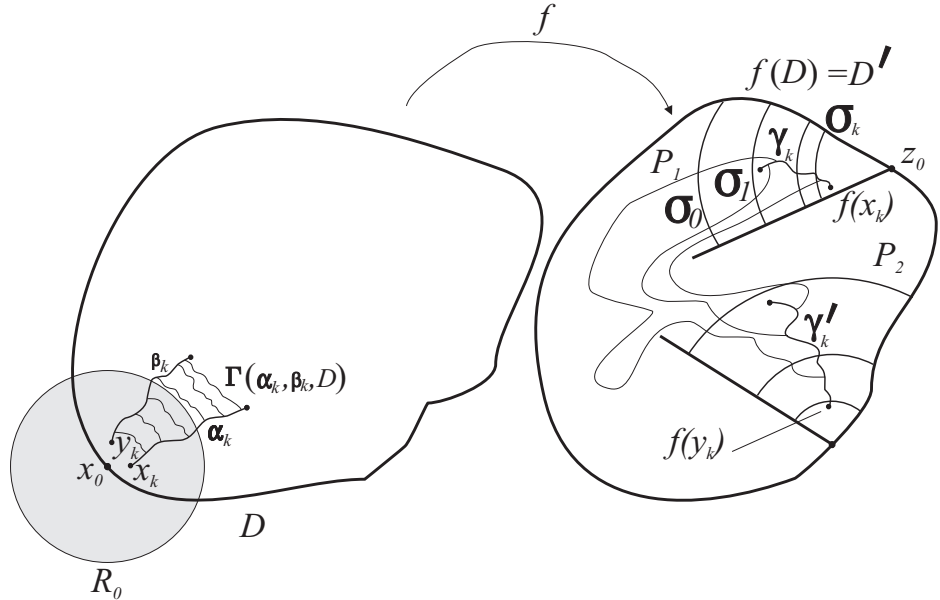}} \caption{To
proof of Theorem~\ref{th1A}}\label{fig1A}
\end{figure}
Observe that points $f(x_1)$ and $f(y_1)$ cannot have more than a
finite number of pre-images in the domain $D$ under the mapping $f.$
Indeed, by Proposition~2.1 in~\cite{Sev$_1$}, $f$ is proper, so that
$f^{\,-1}(f(x_1))$ is a compact set in $D.$ If there is $z_m\in
f^{\,-1}(f(x_1)),$ $m=1,2,\ldots,$ $z_n\ne z_k$ for $n\ne k,$ then
we may find a subsequence $z_{m_p},$ $p=1,2,\ldots ,$ such that
$z_{m_p}\rightarrow z_0$ as $p\rightarrow \infty$ for some $z_0\in
D.$ The latter contradicts the discreteness of $f.$ Thus,
$f^{\,-1}(f(x_1))$ is finite. Similarly, $f^{\,-1}(f(y_1))$ is
finite. Thus there is $R_0>0$ such that $\alpha_k(1), \beta_k(1)\in
D\setminus B(x_0, R_0)$ for any $k=1,2,\ldots .$ Since $D$ has a
weakly flat boundary, for any $P>0$ there is $k=k_P\geqslant 1$ such
that
\begin{equation}\label{eq7}
M_{\alpha}(\Gamma(|\alpha_k|, |\beta_k|,
D))>P\qquad\forall\,\,k\geqslant k_P\,.
\end{equation}
We show that the condition~(\ref{eq7}) contradicts~(\ref{eq2*A}).
Indeed, since $f(\Gamma(|\alpha_k|, |\beta_k|, D))\subset\Gamma(D_0,
D_*, D^{\,\prime})$ and, in addition, by~(\ref{eq11}) $$\Gamma(D_0,
D_*, D^{\,\prime})>\Gamma(S(y_0, r_0/2), S(y_0, r_0), A(y_0, r_0/2,
r_0))$$ we obtain that
$$M_{\alpha}(\Gamma(|\alpha_k|, |\beta_k|, D))>\Gamma_f(y_0, r_0/2, r_0)\,.$$
From the last relation, by minorization principle of the modulus
(see, e.g., \cite[Theorem~1(c)]{Fu})
\begin{equation}\label{eq5}
M_{\alpha}(\Gamma(|\alpha_k|, |\beta_k|, D))\leqslant
M_{\alpha}(\Gamma_f(y_0, r_0/2, r_0))\,.
\end{equation}
Set $\eta(t)= \left\{
\begin{array}{rr}
\frac{2}{r_0}, & t\in [r_0/2, r_0],\\
0,  &  t\not\in [r_0/2, r_0]
\end{array}
\right. .$
Observe that $\eta$ satisfies the relation~(\ref{eqA2}) for
$r_1:=r_1$ and $r_2:=r_0.$ Now, by~(\ref{eq2*A}) and~(\ref{eq5}) we
obtain that
\begin{equation}\label{eq11A}
M_{\alpha}(\Gamma(|\alpha_k|, |\beta_k|, D)) \leqslant
\frac{2}{r_0^{\alpha}}\int\limits_{D^{\,\prime}}
Q(y)\,d\mu^{\,\prime}(y):=c<\infty\quad\forall\,\, k\in {\Bbb N}\,,
\end{equation}
because~$Q\in L^1(D^{\,\prime}).$ The relation~(\ref{eq11A})
contradicts the condition~(\ref{eq7}). The contradiction obtained
above refutes the assumption that there is no limit of the mapping
$f$ at the point $x_0.$

It remains to show that
$\overline{f}(\overline{D})=\overline{D^{\,\prime}}_P.$ Obviously
$\overline{f}(\overline{D})\subset\overline{D^{\,\prime}}_P.$ Let us
to show that $\overline{D^{\,\prime}}_P\subset
\overline{f}(\overline{D}).$ Indeed, let $y_0\in
\overline{D^{\,\prime}}_P.$ Then either $y_0\in D^{\,\prime},$ or
$y_0\in E_{D^{\,\prime}}.$ If $y_0\in D^{\,\prime},$ then
$y_0=f(x_0)$ and $y_0\in \overline{f}(\overline{D}),$ because $f$
maps $D$ onto $D^{\,\prime}.$ Finally, let $y_0\in
E_{D^{\,\prime}}.$ Since $D^{\,\prime}_P$ is compact, there is a
sequence $y_k\in D^{\,\prime}$ such that $m_P(y_k, y_0)\rightarrow
0$ as $k\rightarrow\infty,$ $y_k=f(x_k)$ and $x_k\in D,$ where $m_P$
is defined in~(\ref{pr2}). Since $\overline{D}$ is compact, we may
assume that $x_k\rightarrow x_0,$ where $x_0\in\overline{D}.$ Note
that $x_0\in
\partial D,$ because $f$ is open. Thus
$f(x_0)=y_0\in \overline{f}(\partial D)\subset
\overline{f}(\overline{D}).$ Theorem is completely proved.~$\Box$

\section{Auxiliary lemmas}

Our immediate goal is the equicontinuity (global behavior) of
mappings in the closure of a metric space. We note that for mappings
with the direct Poletsky inequality, similar results were
established in~\cite[section~5]{Sev$_1$}. Our goal is to investigate
similar mappings with the inverse Poletsky inequality. For this
purpose, we use the approach taken when considering similar issues
in Euclidean space, see~\cite{ISS}.

\medskip
A homeomorphism $f: D\rightarrow {\Bbb R}^n,$ $D\subset X,$ is
called a {\it quasiconformal mapping} if there is $K\geqslant 1$
such that
\begin{equation}\label{eq1}
M_n(f(\Gamma))\leqslant K\cdot M_{\alpha}(\Gamma)
\end{equation}
holds for any family of paths $\Gamma$ in $D,$ where $\alpha$ is a
Hausdorff dimension of $D.$ We say that the boundary of the domain
$D$ in ${\Bbb R}^n$ is {\it locally quasiconformal}, if each point
$x_0\in\partial D$ has a neighborhood $U$ in ${\Bbb R}^n$, which can
be mapped by a quasiconformal mapping $\varphi$ onto the unit ball
${\Bbb B}^n\subset{\Bbb R}^n$ so that $\varphi(\partial D\cap U)$ is
the intersection of ${\Bbb B}^n$ with the coordinate hyperplane.  We
say that a bounded domain $D$ in ${\Bbb R}^n$ is {\it regular}, if
$D$ can be quasiconformally mapped to a domain with a locally
quasiconformal boundary whose closure is a compact in ${\Bbb R}^n.$
This definition is slightly different from that given in~\cite{ISS}.
The following statement is true, see
e.g.~\cite[Proposition~1]{SevSkv$_1$}.

\medskip
\begin{proposition}\label{pr1}{\sl\,
Let $n\geqslant 2, $ and let $D$ be a domain in ${\Bbb R}^n$ that is
locally connected on its boundary. Then every two pairs of points
$a\in D, b\in \overline{D}$ and $c\in D, d\in \overline{D}$ can be
joined by non-intersecting paths $\gamma_1:[0, 1]\rightarrow
\overline{D}$ and $\gamma_2:[0, 1]\rightarrow \overline{D}$ so that
$\gamma_i(t)\in D$ for all $t\in (0, 1)$ and all $i=1,2,$ while
$\gamma_1(0)=a,$ $\gamma_1(1)=b,$ $\gamma_2(0)=c$ and
$\gamma_2(1)=d.$}
\end{proposition}

\medskip
The proof of the following statement completely repeats the proof
of~\cite[Theorem~17.10]{Va}, and therefore is omitted.

\begin{proposition}\label{pr2A}
{\sl\, Let $D\subset {\Bbb R}^n$ be a domain with a locally
quasiconformal boundary, then the boundary of this domain is weakly
flat. Moreover, the neighborhood of $U$ in the definition of a
locally quasiconformal boundary can be taken arbitrarily small, and
in this definition $\varphi(x_0)=0.$ }
\end{proposition}

\medskip
The following statement holds.

\medskip
\begin{lemma}\label{lem2}{\sl\,Let $D^{\,\prime}$ be a bounded regular domain in $X$
with a finite Hausdorff dimension $\alpha^{\,\prime}$ which is
finitely connected at the boundary and, besides that,
$\overline{D^{\,\prime}}$ is compact. Assume that $X^{\,\prime}$ is
complete and supports $\alpha^{\,\prime}$-Poincar\'{e} inequality,
and that the measure $\mu^{\,\prime}$ is doubling. Let $h$ be a
quasiconformal mapping of $D^{\,\prime}$ onto $D_0\subset {\Bbb
R}^n,$ where $D_0$ is a domain with a locally quasiconformal
boundary. If $P\in E_{D^{\,\prime}},$ then $h(P)\in E_{D_0}.$}
\end{lemma}

\medskip
\begin{proof}
Let $P=\left\{E_k\right\}_{k=1}^{\infty},$ where $E_k,$
$k=1,2,\ldots ,$ is a corresponding chain of acceptable sets, so
that $E_k$ is connected set such that $E_k \varsubsetneq
D^{\,\prime},$ $\overline{E_k}\cap\partial
D^{\,\prime}\ne\varnothing,$ in addition,

\medskip
1. $E_{k+1}\subset E_k$ for all $k=1, 2,\ldots,$

2. ${\rm dist}\,(D^{\,\prime}\cap \partial E_{k+1}, D^{\,\prime}\cap
\partial E_k)>0$ for all $k=1, 2,\ldots,$

3. $I(P):=\bigcap\limits_{k=1}^{\infty}\overline{E_k}\subset
\partial D^{\,\prime}.$\label{prop3}

\medskip
Let us show, first of all, that the above homeomorphism $h$
preserves the mentioned properties 1--3. Since $h$ is a
homeomorphism, $h(E_k)$ is connected for any $k=1,2,\ldots $ (see
Theorem~3 in \cite[Ch.~5, $\S\,$46, item I]{Ku}), $h(E_k)\ne D_0$
and $\overline{h(E_k)}\cap\partial D_0\ne\varnothing$ (see
\cite[Proposition~13.5]{MRSY}). The relations $h(E_{k+1})\subset
h(E_k),$ $k=1, 2,\ldots,$ are obvious. Let us to prove that
\begin{equation}\label{eq4E}
{\rm dist}\,(D_0\cap
\partial h(E_{k+1}), D_0\cap \partial h(E_k))>0 \quad {\rm for\,\,\,any\,\,\,} k=1, 2,\ldots
\,.
\end{equation}
Suppose the opposite, namely, let
\begin{equation}\label{eq3D}
{\rm dist\,}(D_0\cap
\partial h(E_k), D_0\cap \partial h(E_{k+1}))=0\,.
\end{equation}
Set
$$\rho(x)=\begin{cases}
\frac{1}{{\rm dist\,}(D^{\,\prime}\cap \partial E_k, D^{\,\prime}\cap \partial E_{k+1})}\,,& x\in D^{\,\prime}\,,\\
0\,,&{\rm otherwise}\,.
\end{cases}
$$
Let $\Gamma=\Gamma(D^{\,\prime}\cap \partial E_k, D^{\,\prime}\cap
\partial E_{k+1}, D^{\,\prime}).$ Obviously that, $\rho\in {\rm adm}\,\Gamma$
and by the definition of modulus in~(\ref{eq13.5})
\begin{equation}\label{eq1D}
M_{\alpha^{\,\prime}}(\Gamma)\leqslant \frac{1}{\left({\rm
dist\,}(D^{\,\prime}\cap
\partial E_k, D^{\,\prime}\cap \partial
E_{k+1})\right)^{\alpha^{\,\prime}}}\cdot\mu(D^{\,\prime})<\infty\,.
\end{equation}
Here we took into account that $\overline{D^{\,\prime}}$ is compact
and $\mu(B)>0$ for any ball $B\subset X.$ On the other hand, since
$h$ is a homeomorphism we obtain that
\begin{equation}\label{eq2A}
M_n(\Gamma(D_0\cap \partial h(E_k), D_0\cap
\partial h(E_{k+1}), D_0))=M_n(h(\Gamma))\,.
\end{equation}
It follows from~(\ref{eq3D}) that there are $x_m\in D_0\cap \partial
h(E_k)$ and $y_m\in D_0\cap
\partial h(E_{k+1})$ such that $d(x_m, y_m)\rightarrow 0$ as
$m\rightarrow \infty.$ Since $D_0$ is bounded, we may consider that
$x_m\rightarrow x_0\in \overline{D_0\cap \partial h(E_k)}$ as
$m\rightarrow\infty$ and $y_m\rightarrow x_0\in \overline{D_0}$ as
$m\rightarrow\infty.$ Now, by the triangle inequality, $d(x_0,
y_0)\leqslant d(x_0, x_m)+d(x_k, y_m)+d(y_m, y_0)\rightarrow 0$ as
$m\rightarrow\infty.$ Thus, $x_0=y_0\in \overline{D_0\cap \partial
h(E_k)}\cap \overline{D_0\cap
\partial h(E_{k+1})}.$ By Proposition~\ref{pr2A}, $D_0$ has a weakly
flat boundary. By~\cite[Theorem~17.10]{Va}, the modulus of families
of paths joining the sets with a common point in a domain with a
weakly flat boundary equals to infinity. Thus, by~(\ref{eq2A})
\begin{equation}\label{eq1E}
M_n(h(\Gamma))=\infty\,.
\end{equation}
The relation~(\ref{eq1E}) contradicts with~(\ref{eq1D}), because, by
the quasiconformality of $h,$ we obtain that
$$M_n(h(\Gamma))\leqslant K\cdot M_{\alpha}(\Gamma)<\infty\,.$$
Thus, (\ref{eq4E}) holds, as required.

\medskip
It remains to show the validity of property 3 on a
page~\pageref{prop3} for the mapped family of acceptable
sets~$h(E_k),$ $k=1,2,\ldots .$ Let $y\in
\bigcap\limits_{k=1}^{\infty}\overline{h(E_k)},$ then $y\in
\overline{h(E_k)}$ for any $k\in {\Bbb N}.$ Now,
$y=\lim\limits_{l\rightarrow \infty} y^k_l,$ where $y^k_l=h(x^k_l),$
and $x^k_l\in E_k,$ $l=1,2,\ldots .$ Therefore, for any $k\in {\Bbb
N}$ there is a number $l_k,$ $k=1,2,\ldots ,$ such that
$|h(x^k_{l_k})-y|<1/2^k.$ Since $\overline{E_1}$ is compact (see
\cite{ABBS}), we may consider that $x^k_{l_k}$ converge to some
$x_0\in \overline{E_1}$ as $k\rightarrow\infty.$ Thus, $y\in
\partial D_0$ (see \cite[Proposition~13.5]{MRSY}) and, consequently,
$\bigcap\limits_{k=1}^{\infty}\overline{h(E_k)}\subset \partial
D_0.$ Thus, the chain of cuts $h(E_k),$ $k=1,2\ldots ,$ defines some
end $h(P).$

\medskip
Finally, let us show that the set $I(P)$ consists of exactly one
point. To do this, we show that the mapping $h$ is the so-called
ring $Q$-mapping with $Q:=K$ (see, e.g., \cite{Sev$_1$}). Let $0<
r_1<r_2<\infty$ and let $\eta:(r_1, r_2)\rightarrow [0, \infty]$ be
a Lebesgue measurable function  such that
$\int\limits_{r_1}^{r_2}\eta(r)\,dr\geqslant 1,$ let $x_0\in
\partial D^{\,\prime},$ $S_i=S(x_0, r_i)$ and let $A=A(x_0, r_1, r_2)=\{x\in X:
r_1<d(x, x_0)<r_2\}. $ Put $\Gamma=\Gamma(S_1, S_2, A\cap
D^{\,\prime}).$ Set
$$\rho(x)=\begin{cases}\eta(d(x, x_0)),& x\in A\cap D^{\,\prime}, \\
0\,,&{\rm otherwise}
 \end{cases}\,.$$
Let $\gamma$ be a locally rectifiable path in $\Gamma.$ Then
by~\cite[Proposition~13.4]{MRSY}
$$\int\limits_{\gamma}\rho(x)\,|dx|\geqslant \int\limits_{r_1}^{r_2}\eta(r)\,dr\geqslant 1\,.$$
So, $\rho\in {\rm adm}\,\Gamma$ and, consequently, by the definition
of $h$ in~(\ref{eq1}) and by the definition of the modulus of
families of paths in~(\ref{eq13.5})
\begin{equation}\label{eq1G}
M_n(h(\Gamma))\leqslant K\cdot
M_{\alpha^{\,\prime}}(\Gamma)\leqslant\int \limits_A
K\eta^{\alpha^{\,\prime}}(d(x, x_0))\,d\mu(x)\,.
\end{equation}
The latter means that $h$ is a ring $K$-mapping at $x_0,$ as
required. Since $D_0$ has a quasiconformal boundary, $\partial D_0$
is weakly flat by Proposition~\ref{pr2A} and, therefore, is strongly
accessible (see e.g. \cite[Proposition~13.6]{MRSY}). Besides that,
$\overline{D_0}$ is a compactum by the assumption of the lemma.
Then, by~\cite[Theorem~1.1]{Sev$_1$} $h$ has a continuous extension
$h:\overline{D^{\,\prime}}_P\rightarrow \overline{D_0}.$

\medskip
Let now $y\in
I(h(P)):=\bigcap\limits_{k=1}^{\infty}\overline{h(E_k)},$ then by
proving above $y=\lim\limits_{k\rightarrow\infty}h(x^k_{l_k}),$
where $x^k_{l_k}$ is some sequence converging  to $P$ as
$k\rightarrow\infty.$ Since $h$ has a continuous extension
$h:\overline{D^{\,\prime}}_P\rightarrow \overline{D_0},$ $y$ is a
one-point set, as required. Now, $h(P)$ is a prime end (see
e.g.~\cite[Proposition~7.1]{ABBS}).~$\Box$
\end{proof}

\medskip
Let $D^{\,\prime}$ be a domain in a locally connected space $X$ and
let $a\in D^{\,\prime}.$ Then we may define a sequence $V_k,$
$k=1,2,\ldots,$ of neighborhoods of a point $a$ such that
$V_{k+1}\subset V_k,$ ${\rm dist}\,(\partial V_{k+1}, \partial
V_k)>0$ and $\bigcap\limits_{k=1}^{\infty}\overline{E_k}=a.$ Two
arbitrary such sequences $\{V_k\}_{k=1}^{\infty}$ and
$\{U_k\}_{k=1}^{\infty}$ will be considered equivalent. In what
follows, by a ''prime end'' corresponding to the point $a,$ we mean
the equivalence class of the sequences of ''admissible sets'' $V_k,$
$k=1,2,\ldots .$ indicated above. Let us establish the following
statement (see, for example, \cite[Lemma~2.1]{ISS}).

\medskip
\begin{lemma}\label{lem1}{\sl\, Let $D^{\,\prime}$ be a bounded regular domain in $X^{\,\prime}$
with a finite Hausdorff dimension $\alpha^{\,\prime}$ which is
finitely connected at the boundary. Assume that $X^{\,\prime}$ is
complete and supports $\alpha^{\,\prime}$-Poincar\'{e} inequality,
and that the measure $\mu$ is doubling. Let $x_m\rightarrow P_1,$
$y_m\rightarrow P_2$ as $m\rightarrow\infty,$ $P_1, P_2\in
\overline{D^{\,\prime}}_P,$ $P_1\ne P_2.$ Suppose that $d_m, g_m,$
$m=1,2,\ldots,$ are sequences of decreasing domains, corresponding
to $P_1$ and $P_2,$ $d_1\cap g_1=\varnothing,$ and $x_0, y_0\in
D^{\,\prime}\setminus (d_1\cup g_1).$ Then there are arbitrarily
large $k_0\in {\Bbb N},$ $M_0=M_0(k_0)\in {\Bbb N}$ and
$0<t_1=t_1(k_0), t_2=t_2(k_0)<1$ for which the following condition
is fulfilled: for each $m\geqslant M_0$ there are
non-intersecting paths
$$\gamma_{1,m}(t)=\quad\left\{
\begin{array}{rr}
\widetilde{\alpha}(t), & t\in [0, t_1],\\
\widetilde{\alpha_m}(t), & t\in [t_1, 1]\end{array}
\right.\,,\quad\gamma_{2,m}(t)=\quad\left\{
\begin{array}{rr}
\widetilde{\beta}(t), & t\in [0, t_2],\\
\widetilde{\beta_m}(t), & t\in [t_2, 1]\end{array}\,, \right.$$
such that:

1) $\gamma_{1, m}(0)=x_0,$ $\gamma_{1, m}(1)=x_m,$ $\gamma_{2,
m}(0)=y_0$ and $\gamma_{2, m}(1)=y_m;$

2) $|\gamma_{1, m}|\cap \overline{g_{k_0}}=\varnothing=|\gamma_{2,
m}|\cap \overline{d_{k_0}};$

3) $\widetilde{\alpha_m}(t)\in d_{k_0+1}$ for $t\in [t_1, 1]$ and
$\widetilde{\beta_m}(t)\in g_{k_0+1}$ for $t\in [t_2, 1]$ (see
Figure~\ref{fig3}).
\begin{figure}[h]
\centering\includegraphics[scale=0.7]{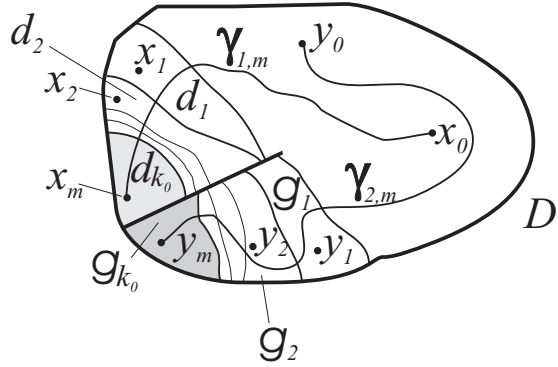} \caption{To the
statement of Lemma~\ref{lem1}}\label{fig3}
\end{figure}
}
\end{lemma}

\begin{proof}
Let $P_1, P_2\in E_D^{\,\prime}.$ Since space $X^{\,\prime}$ is
complete and supports a $\alpha^{\,\prime}$-Poincar\'{e} inequality,
$X^{\,\prime}$ is proper and locally connected (see
\cite[Section~2]{ABBS} and \cite[Proposition~3.1]{BB}). Then, by
Remarks~2.6 and~4.5 in \cite{ABBS} we may consider that the sets
$d_k$ and $g_k,$ $k=1,2,\ldots ,$ are open and path connected. If
$P_1\in D^{\,\prime}$ or $P_2\in D^{\,\prime},$ then the sets $d_k$
and $g_k,$ $k=1,2,\ldots ,$ are well defined, in addition, we may
consider that $d_k$ and $g_k$ are open and path connected, as well.

\medskip
We use a slightly modified scheme for the proof of Lemma~2.1
in~\cite{ISS}. Since, by condition, $D^{\,\prime}$ is a regular
domain, it can be mapped onto some domain with a locally
quasiconformal boundary by (some) quasiconformal mapping
$h:D^{\,\prime}\rightarrow D_0.$ Note that the domain $D_0$ is
locally connected on its boundary, which follows directly from the
definition of local quasiconformality. Let $P_1, P_2\in
E_D^{\,\prime}.$ Then by Lemma~\ref{lem2} $h(P_1)$ and $h(P_2)$ are
prime ends in $E_{D_0}.$ Observe that $I(P_1)$ and $I(P_2)$ are
different points $a$ and $b$ in $\partial D_0$ whenever $P_1\ne P_2$
(see e.g. \cite[Lemma~4.1]{Sev$_1$}). If $P_1$ or $P_2$ are inner
points of $D^{\,\prime},$ then $h(P_1)$ (or $h(P_2)$) are inner
points of $D_0,$ which we denote by $a$ or $b,$ respectively. Since
by assumption $x_0, y_0\in D^{\,\prime}\setminus (d_1\cup g_1),$
then, in particular, $P_1\ne x_0\ne P_2,$ $P_1\ne y_0\ne P_2.$ This
implies that $a, b, h(x_0)$ and $h(y_0)$ are four different points
in $\overline{D_0},$ at least two of which are inner points of $D_0$
(see Figure~\ref{fig4}).
\begin{figure}[h]
\centering\includegraphics[width=300pt]{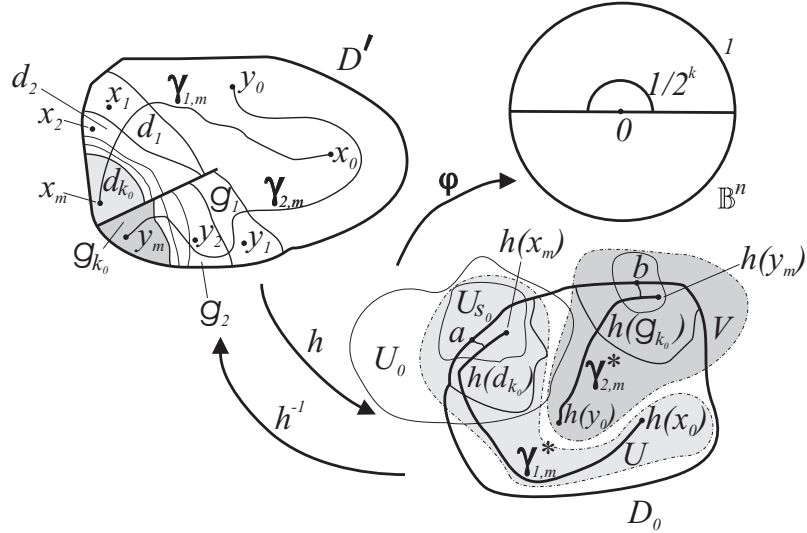} \caption{To
the proof of Lemma~\ref{lem1}}\label{fig4}
\end{figure}
By Proposition~\ref{pr1}, one can join pairs of points $a, h(x_0)$
and $b, h(y_0)$ by disjoint paths $\alpha:[0, 1]\rightarrow
\overline{D_0}$ and $\beta:[0, 1]\rightarrow \overline{D_0}$ so that
$|\alpha|\cap |\beta|=\varnothing,$ $\alpha(t),\beta(t)\in D_0$ for
all $t\in (0, 1),$ $\alpha(0)=h(x_0),$ $\alpha(1)=a,$
$\beta(0)=h(y_0)$ and $\beta(1)=b.$ Since ${\Bbb R}^n$  is a normal
topological space, the loci $|\alpha|$ and $|\beta|$ have
non-intersecting open neighborhoods $U, V$ such that
\begin{equation}\label{eq3B}
|\alpha|\subset U,\quad |\beta|\subset V\,.
\end{equation}
Here two cases are possible: either $h(P_1)$ is a prime end in
$E_{D_0},$ or a point in  $D_0.$ Let $h(P_1)$ be a prime end in
$E_{D_0}.$ Since $I(h(P_1))=a,$ then there is a number $k_1\in {\Bbb
N}$ such that $\overline{h(d_k)}\subset U$ for $k\geqslant k_1.$ If
$h(P_1)$ is a point of $D_0,$ then there is also a number $k_1\in
{\Bbb N}$ such that $\overline{h(d_k)}\subset U$ for all $k\geqslant
k_1.$ In either of these two cases, $\overline{h(d_k)}\subset U$ for
$k\geqslant k_1.$ Similarly, there is a number $k_2\in {\Bbb N}$
such that $\overline{h(g_k)}\subset V$ for all $k\geqslant k_2.$
Then for $k_0:=\max\{k_1, k_2\}$ we obtain that
\begin{equation}\label{eq3}
\overline{h(d_k)}\subset U\,,\quad \overline{h(g_k)}\subset V\,,
\quad U\cap V=\varnothing\,, \quad k\geqslant k_0\,.
\end{equation}
Since the sequence $x_m$ converges to $P_1$ as $m\rightarrow\infty,$
then $h(x_m)$ converges to $a.$ Therefore, there is a number $m_1\in
{\Bbb N}$ such that $h(x_m)\in h(d_{k_0+1})$ for $m\geqslant m_1.$
Similarly, since $y_m$ converges to $P_2$ as $m\rightarrow\infty,$
then $h(y_m)$ converges to $b.$ Therefore, there is a number $m_2\in
{\Bbb N}$ such that $h(y_m)\in h(g_{k_0+1})$ for $m\geqslant m_2.$
Put $M_0:=\max\{m_1, m_2\}.$ Show that
\begin{equation}\label{eq2}
|\alpha|\cap h(d_{k_0+1})\ne\varnothing,\qquad |\beta|\cap
h(g_{k_0+1})\ne\varnothing\,.
\end{equation}
It suffices to establish the first of these relations, since the
second relation can be proved similarly. If $a=h(P_1)$ is an inner
point of $D_0,$ then this inclusion is obvious. Now suppose that
$h(P_1)\in E_{D_0}.$ Since $D_0$ has a locally quasiconformal
boundary, there is a sequence of spheres $S(0,1/2^k),$
$k=0,1,2,\ldots,$ a decreasing sequence of neighborhoods $U_k$ of
the point $a$ and some quasiconformal mapping
$\varphi:U_0\rightarrow {\Bbb B}^n,$ for which $\varphi(U_k)=B(0,
1/2^k),$ $\varphi(\partial U_k\cap D_0)=S(0, 1/2^k)\cap {\Bbb
B}^n_+,$ where ${\Bbb B}^n_+=\{x=(x_1, \ldots, x_n): |x|<1, x_n>0\}$
(see the arguments given in the proof of~\cite[Lemma~3.5]{Na}). Note
that $U_k\cap D_0$ is a domain, since $U_k\cap
D_0=\varphi^{\,-1}(B_+(0, 1/2^k)),$ $B_+(0, 1/2^k)=\{x=(x_1, \ldots,
x_n): |x|<1/2^k, x_n>0\},$ and $\varphi$ is a homeomorphism.

\medskip
Observe that the sequence of domains $U_k\cap D_0$ corresponds to
some prime end, the impression of which is the point $a.$ Moreover,
since $D_0$ is a domain with a locally quasiconformal boundary,
$D_0$ is locally connected on the boundary and, therefore, the prime
end $h(P_1)$ with an impression $a$ a is unique (see
\cite[Corollary~10.14]{ABBS}). Therefore any domain $h(d_m)$
contains all domains $U_k\cap D_0,$ except for a finite number, and
vice versa. In particular, there is $s_0\in {\Bbb N}$ such that
$U_k\cap D_0\subset h(d_{k_0+1})$ for all $k\geqslant s_0.$ Since
$a\in |\alpha|,$ there is $t_1\in (0,1)$ such that
$p:=\alpha(t_1)\in U_{s_0}\cap D_0.$ But then also $p\in
h(d_{k_0+1}),$ since $U_{s_0}\cap D_0\subset h(d_{k_0+1}).$ The
first relation in~(\ref{eq2}) is proved. As we said above, the
second relation may be proved in exactly the same way.

\medskip
So, let $p:=\alpha(t_1)\in |\alpha|\cap h(d_{k_0+1}).$ Fix
$m\geqslant M_0$ and join the point $p$ with the point $h(x_m)$
using the path $\alpha_m:[t_1, 1]\rightarrow h(d_{k_0+1})$ so that
$\alpha_m(t_1)=p,$ $\alpha_m(1)=h(x_m),$ what is possible because
$h(d_{k_0+1})$ is a domain.
Set
\begin{equation}\label{eq10AA}
\gamma^{\,*}_{1,m}(t)=\quad\left\{
\begin{array}{rr}
\alpha(t), & t\in [0, t_1],\\
\alpha_m(t), & t\in [t_1, 1]\end{array} \right.\,.
\end{equation}

Note that the path $\gamma^{\,*}_{1,m}$ completely lies in $U.$
Reasoning similarly, we have the point $t_2\in (0, 1)$ and the point
$q:=\beta(t_2)\in |\beta|\cap h(g_{k_0+1}).$ Fix $m\geqslant M_0$
and join the point $q$ with the point $h(y_m)$ using the path
$\beta_m:[t_2, 1]\rightarrow h(g_{k_0+1})$ so that $\beta_m(t_2)=q,$
$\beta_m(1)=h(y_m),$ that is possible, because $h(g_{k_0+1})$ is a
domain.
Set
\begin{equation}\label{eq10B}
\gamma^{\,*}_{2,m}(t)=\quad\left\{
\begin{array}{rr}
\beta(t), & t\in [0, t_2],\\
\beta_m(t), & t\in [t_2, 1]\end{array} \right.\,.
\end{equation}
Note that the path $\gamma^{\,*}_{2,m}$ completely lies in $V.$ Set
\begin{equation}\label{eq11AA}
\gamma_{1,m}:=h^{\,-1}(\gamma^{\,*}_{1,m})\,,\quad
\gamma_{2,m}:=h^{\,-1}(\gamma^{\,*}_{2,m})\,.
\end{equation}
Note that the paths $\gamma_{1,m}$ and $\gamma_{2,m}$ satisfy all
the conditions of Lemma~\ref{lem1} for $m\geqslant M_0.$ In fact, by
definition, these paths join the points $x_m, x_0$ and $y_m, y_0,$
respectively. The paths $\gamma_{1,m}$ and $\gamma_{2,m}$ do not
intersect, since their images under the mapping $h$ belong to
non-intersecting neighborhoods $U$ and $V,$ respectively.

Note also that $|\gamma_{1,m}|\cap \overline{g_{k_0}}=\varnothing$
for $m\geqslant M_0.$ Indeed, if $x\in |\gamma_{1,m}|\cap
\overline{g_{k_0}},$ then either $x\in |\gamma_{1,m}|\cap g_{k_0}$
or $x\in |\gamma_{1,m}|\cap \partial g_{k_0}.$ In the first case, if
$x\in |\gamma_{1,m}|\cap g_{k_0}$ then $h(x)\in
|\gamma^{\,*}_{1,m}|\cap h(g_{k_0})\subset U\cap h(g_{k_0}),$ which
is impossible due to the relation~(\ref{eq3}). In the second case,
if $x\in |\gamma_{1,m}|\cap \partial g_{k_0},$ then there is a
sequence $z_m\in g_{k_0}$ such that $z_m\rightarrow x$ as
$m\rightarrow\infty.$ Now $h(z_m)\rightarrow h(x)$ as
$m\rightarrow\infty$ and, therefore, $h(x)\in
\overline{h(g_{k_0})}.$ At the same time, $h(x)\in U,$ and this is
impossible by virtue of relation~(\ref{eq3}). Thus, the relation
$|\gamma_{1,m}|\cap \overline{g_{k_0}}=\varnothing$ for $m\geqslant
M_0$ is established.

Similarly, $|\gamma_{2,m}|\cap \overline{d_{k_0}}=\varnothing$ for
$m\geqslant M_0.$ Finally, defining paths $\widetilde{\alpha},$
$\widetilde{\alpha}_m,$ $\widetilde{\beta}$ and
$\widetilde{\beta}_m$ by means of relations
$\widetilde{\alpha}=h^{\,-1}(\alpha),$
$\widetilde{\alpha}_m=h^{\,-1}(\alpha_m),$
$\widetilde{\beta}=h^{\,-1}(\beta)$ and
$\widetilde{\beta}_m=h^{\,-1}(\beta_m),$ we see that these paths
correspond to the construction of $\gamma_{1,m}$ and $\gamma_{2,m},$
and also satisfy conditions~3) from the formulation of the lemma.
Lemma~\ref{lem1} is proved.~$\Box$
\end{proof}

\medskip
Consider the family of paths joining $|\gamma_{1, m}|$ and
$|\gamma_{2, m}|$ in $D^{\,\prime}$ from the previous lemma. The
following statement contains the upper estimate of the modulus of
the transformed family of paths under the mapping $f$ with the
inequality~(\ref{eq2*A}).

\medskip
\begin{lemma}\label{lem4A}
{\sl Let $D$ and $D^{\,\prime}$ be domains with finite Hausdorff
dimensions $\alpha$ and $\alpha^{\,\prime}\geqslant 2$ in spaces
$(X,d,\mu)$ and $(X^{\,\prime},d^{\,\prime}, \mu^{\,\prime}),$
respectively. Assume that $X$ is locally connected, $\overline{D}$
is compact, $X^{\,\prime}$ is complete and supports
$\alpha^{\,\prime}$-Poincar\'{e} inequality, and that the measures
$\mu$ and $µ^{\,\prime}$ are doubling. Let $D^{\,\prime}\subset
X^{\,\prime}$ be a regular bounded domain which is finitely
connected at the boundary, and let $Q:X^{\,\prime}\rightarrow (0,
\infty)$ be integrable function in $D^{\,\prime},$ $Q(y)\equiv 0$
for $y\in X^{\,\prime}\setminus D^{\,\prime}.$ Suppose that
$f:D\rightarrow D^{\,\prime},$ $D^{\,\prime}=f(D),$ is an open
discrete and closed mapping satisfying the relation~(\ref{eq2*A}) at
any point $y_0\in \overline{D^{\,\prime}}.$ Then under conditions
and notation of Lemma~\ref{lem1} we may choose a number $k_0\in
{\Bbb N}$ for which there is $0<N=N(k_0, \Vert Q\Vert_1,
D^{\,\prime})<\infty,$ independent on $m$ and $f,$ such that
$$M_{\alpha}(\Gamma_m)\leqslant N,\qquad m\geqslant M_0=M_0(k_0)\,,$$
where $\Gamma_m$ is a family of paths $\gamma:[0, 1]\rightarrow D$
such that $f(\gamma)\in \Gamma(|\gamma_{1, m}|, |\gamma_{2, m}|,
D^{\,\prime}).$ }
\end{lemma}

\medskip
\begin{proof}
Denote $y_0:=I([E_k])$ (see~\cite[Theorem~10.8]{ABBS}). Arguing as
in the proof of Lemma~2.1 in \cite{Sev$_1$}, we may show that, for
every $r>0$ there exists $N\in {\Bbb N}$ such that
\begin{equation}\label{eq10C} d_k\subset
B(y_0, r)\cap D^{\,\prime}\quad \forall\,\, k\geqslant N\,.
\end{equation}
Since $D^{\,\prime}$ is connected and $d_{k_0}\ne D^{\,\prime},$ we
obtain that $\partial d_{k_0}\cap D^{\,\prime}\ne\varnothing$ (see
\cite[Ch.~5, $\S\,$46, item I]{Ku}). Set $r_0:=d(y_0,
\partial d_{k_0+1}\cap D^{\,\prime}).$ Since $\overline{d_{k_0}}$ is compact, $r_0>0.$
By (\ref{eq10C}), there exists $k_1\in {\Bbb N}$ such that
\begin{equation}\label{eq10D} d_k\subset
B(y_0, r_0/2)\cap D^{\,\prime}\quad \forall\,\, k\geqslant k_1\,.
\end{equation}
Set $D_0:=d_{k_1+1},$ $D_*:=g_{k_1+1}.$ Let us to show that
\begin{equation}\label{eq11D}
\Gamma(D_0, D_*, D^{\,\prime})>\Gamma(S(y_0, r_0/2), S(y_0, r_0),
A(y_0, r_0/2, r_0))\,,
\end{equation}
where $A(y_0, r_1, r_2)=\{y\in X^{\,\prime}: r_0/2<d^{\,\prime}(y,
y_0)<r_0\}.$
Assume that $\Delta\in \Gamma(D_0, D_*, D^{\,\prime}),$ $\Delta:[0,
1]\rightarrow D^{\,\prime}.$ Set
$$|\Delta|:=\{y\in D^{\,\prime}: \exists\,t\in[0, 1]:
\gamma(t)=y\}\,.$$
Since $d_1\cap g_1=\varnothing,$ $|\Delta|\cap
d_{k_0+1}\ne\varnothing\ne |\Delta|\cap (D^{\,\prime}\setminus
d_{k_0+1}).$ Thus,
\begin{equation}\label{eq13A}
|\Delta|\cap \partial d_{k_0+1}\ne\varnothing
\end{equation} (see \cite[Theorem 1, $\S\,$46, item I]{Ku}).
Moreover, observe that
\begin{equation}\label{eq14D}
\Delta(1)\not\in \partial d_{k_0+1}\,.
\end{equation}
Suppose the contrary, i.e., that $\Delta(1)\in \partial d_{k_0+1}.$
By definition of prime end, $\partial d_{k_0+1}\cap D\subset
\overline{d_{k_0}}.$ Since ${\rm dist}\,(D\cap \partial d_{k+1},
D\cap \partial d_k)>0$ for any $k=1, 2,\ldots,$ we obtain that
$\partial d_{k_0+1}\cap D\subset d_{k_0}.$ Therefore, $\Delta(1)\in
d_{k_0}$ and, simultaneously, $\Delta(1)\in g_{k_1+1}\subset
g_{k_0}.$ The last relations contradict with $d_1\cap
g_1=\varnothing.$ Thus, (\ref{eq14D}) holds, as required.

\medskip
By (\ref{eq10D}), we obtain that $|\Delta|\cap B(y_0,
r_0/2)\ne\varnothing.$ We prove that $|\Delta|\cap
(D^{\,\prime}\setminus B(y_0, r_0/2))\ne\varnothing.$ In fact, if it
is not true, then $\Delta(t)\in B(y_0, r_0/2)$ for every $t\in [0,
1].$ However, by (\ref{eq13A}) we obtain that $(\partial
d_{k_0+1}\cap D^{\,\prime})\cap B(y_0, r_0/2)\ne \varnothing,$ that
contradicts to the definition of $r_0.$ Thus, $|\Delta|\cap
(D^{\,\prime}\setminus B(y_0, r_0/2))\ne\varnothing,$ as required.
Now, by \cite[Theorem 1, $\S\,$46, item I]{Ku}, there exists $t_1\in
(0, 1]$ with $\Delta(t_1)\in S(y_0, r_0/2).$ We may consider that
$t_1=\max\{t\in [0, 1]: \Delta(t)\in S(y_0, r_0/2)\}.$ We prove that
$t_1\ne 1.$ Suppose the contrary, i.e., suppose that $t_1=1.$ Now,
we obtain that $\Delta(t)\in B(y_0, r_0/2)$ for every $t\in [0, 1).$
On the other hand, by (\ref{eq13A}) and (\ref{eq14D}), we obtain
that $\partial d_{k_0+1}\cap B(y_0, r_0/2)\ne \varnothing,$ which
contradicts to the definition of $r_0.$ Thus, $t_1\ne 1,$ as
required. Set $\Delta_1:=\Delta|_{[t_1, 1]}.$

\medskip
By the definition, $|\Delta_1|\cap B(y_0, r_0)\ne\varnothing.$ We
prove that $|\Delta_1|\cap (D^{\,\prime}\setminus B(y_0,
r_0))\ne\varnothing.$ In fact, assume the contrary, i.e., assume
that $\gamma_1(t)\in B(y_0, r_0)$ for every $t\in [t_1, 1].$ Since
$\gamma(t)\in B(y_0, r_0/2)$ for $t<t_1,$ by (\ref{eq13A}) we obtain
that $|\gamma_1|\cap
\partial d_{k_0+1}\ne\varnothing.$ Consequently, $B(y_0, r_0)\cap
(\partial d_{k_0+1}\cap D^{\,\prime})\ne\varnothing,$ that
contradicts to the definition of $r_0.$ Thus, $|\Delta_1|\cap
(D^{\,\prime}\setminus B(y_0, r_0))\ne\varnothing,$ as required.
Now, by \cite[Theorem 1, $\S\,$46, item I]{Ku}, there exists $t_2\in
(t_1, 1]$ with $\Delta(t_2)\in S(y_0, r).$ We may consider that
$t_2=\min\{t\in [t_1, 1]: \Delta(t)\in S(y_0, r_0)\}.$ We put
$\Delta_2:=\Delta|_{[t_1, t_2]}.$ Observe that $\Delta>\Delta_2$ and
$\Delta_2\in\Gamma(S(y_0, r_0/2), S(y_0, r_0), A(y_0, r_0/2, r_0)).$
Thus, (\ref{eq11D}) has been proved.

\medskip
Let $k_0$ be an arbitrary number for which the statement of
Lemma~\ref{lem1} holds. Without loss of generality, we may assume
that $k_0\geqslant k_1+1.$ By the definition of $\gamma_{1, m}$ and
of the family $\Gamma_m$ we may write that
\begin{equation}\label{eq7A}
\Gamma_m=\Gamma_m^1\cup \Gamma_m^2\,,
\end{equation}
where $\Gamma_m^1$ is a family of paths $\gamma\in\Gamma_m$ such
that $f(\gamma)\in \Gamma(|\widetilde{\alpha}|, |\gamma_{2, m}|,
D^{\,\prime})$ and $\Gamma_m^2$ is a family of paths
$\gamma\in\Gamma_m$ such that $f(\gamma)\in
\Gamma(|\widetilde{\alpha}_m|, |\gamma_{2, m}|, D^{\,\prime}).$

\medskip
Taking into account the notation of Lemma~\ref{lem1}, we put
$$\varepsilon_0:=\min\{{\rm dist}\,(|\widetilde{\alpha}|,
\overline{g_{k_0}}), {\rm dist}\,(|\widetilde{\alpha}|,
|\widetilde{\beta}|)\}>0\,.$$
Let us consider the covering $\bigcup\limits_{x\in
|\widetilde{\alpha}|}B(x, \varepsilon_0/4)$ of
$|\widetilde{\alpha}|.$ Since $|\widetilde{\alpha}|$ is a compactum
in $D^{\,\prime},$ there are numbers $i_1,\ldots, i_{N_0}$ such that
$|\widetilde{\alpha}|\subset \bigcup\limits_{i=1}^{N_0} B(z_i,
\varepsilon_0/4),$ where $z_i\in |\widetilde{\alpha}|$ for
$1\leqslant i\leqslant N_0.$ By~\cite[Theorem~1.I.5.46]{Ku}
\begin{equation}\label{eq5E}
\Gamma(|\widetilde{\alpha}|, |\gamma_{2, m}|,
D^{\,\prime})>\bigcup\limits_{i=1}^{N_0} \Gamma(S(z_i,
\varepsilon_0/4), S(z_i, \varepsilon_0/2), A(z_i, \varepsilon_0/4,
\varepsilon_0/2))\,.
\end{equation}
Fix $\gamma\in \Gamma_m^1,$ $\gamma:[0, 1]\rightarrow D,$
$\gamma(0)\in |\widetilde{\alpha}|,$ $\gamma(1)\in |\gamma_{2, m}|.$
It follows from~(\ref{eq5E}) that $f(\gamma)$ has a subpath
$f(\gamma)_1:=f(\gamma)|_{[p_1, p_2]}$ such that
$$f(\gamma)_1\in \Gamma(S(z_i, \varepsilon_0/4), S(z_i,
\varepsilon_0/2), A(z_i, \varepsilon_0/4,  \varepsilon_0/2))$$ for
some $1\leqslant i\leqslant N_0.$ Then $\gamma|_{[p_1, p_2]}$ is a
subpath of $\gamma$ and belongs to~$\Gamma_f(z_i, \varepsilon_0/4,
\varepsilon_0/2),$ because
$$f(\gamma|_{[p_1, p_2]})=f(\gamma)|_{[p_1, p_2]}\in\Gamma(S(z_i,
\varepsilon_0/4), S(z_i, \varepsilon_0/2), A(z_i, \varepsilon_0/4,
\varepsilon_0/2)).$$ Thus
\begin{equation}\label{eq6E}
\Gamma_m^1>\bigcup\limits_{i=1}^{N_0} \Gamma_f(z_i, \varepsilon_0/4,
\varepsilon_0/2)\,.
\end{equation}
Put
$$\eta(t)= \left\{
\begin{array}{rr}
4/\varepsilon_0, & t\in [\varepsilon_0/4, \varepsilon_0/2],\\
0,  &  t\not\in [\varepsilon_0/4, \varepsilon_0/2]
\end{array}
\right. \,.$$
Observe that the function~$\eta$ satisfies the
relation~(\ref{eqA2}). Then, by the definition of $f$
in~(\ref{eq2*A}), by the relation~(\ref{eq6E}) and due to the
subadditivity of the modulus of families of paths
(see~\cite[Theorem~6.2]{Va}), we obtain that
\begin{equation}\label{eq1H}
M_{\alpha}(\Gamma_m^1)\leqslant \sum\limits_{i=1}^{N_0}
M_{\alpha}(\Gamma_f(z_i, \varepsilon_0/4, \varepsilon_0/2))\leqslant
\sum\limits_{i=1}^{N_0} \frac{N_04^{\alpha}\Vert
Q\Vert_1}{\varepsilon^{\alpha}_0}\,,\qquad m\geqslant M_0\,,
\end{equation}
where $\Vert Q\Vert_1=\int\limits_{D^{\,\prime}}Q(x)\,dm(x).$

\medskip
Let $\widetilde{\gamma}\in f(\gamma)\in
\Gamma(|\widetilde{\alpha}_m|, |\gamma_{2, m}|, D^{\,\prime}),$
$\widetilde{\gamma}:[0,1]\rightarrow D^{\,\prime},$
$\widetilde{\gamma}(0)\in |\widetilde{\alpha}_m|$ and
$\widetilde{\gamma}(1)\in |\gamma_{2, m}|.$ Since
$\widetilde{\alpha_m}(t)\in d_{k_0+1}$ for $t\in [t_1, 1]$ and
$\widetilde{\beta_m}(t)\in g_{k_0+1}$ for $t\in [t_2, 1],$
by~(\ref{eq11D}) we obtain that
$\Gamma_m^2>\Gamma_f(y_0, r_0/2, r_0).$
Arguing similarly as above, we put
$$\eta(t)= \left\{
\begin{array}{rr}
2/r_0, & t\in [r_0/2, r_0],\\
0,  &  t\not\in [r_0/2, r_0]
\end{array}
\right. \,.$$
Now, by the last relation we obtain that
\begin{equation}\label{eq4DD}
M_{\alpha}(\Gamma_m^2)\leqslant \frac{2\Vert
Q\Vert_1}{r_0^{\alpha}}\,,\quad m\geqslant M_0\,.
\end{equation}
Thus, by~(\ref{eq7A}), (\ref{eq1H}) and~(\ref{eq4DD}), due to the
subadditivity of the modulus of families of paths
(see~\cite[Theorem~1]{Fu}), we obtain that
$$M_{\alpha}(\Gamma_m)\leqslant
\left(\frac{N_04^{\alpha}}{\varepsilon^{\alpha}_0}+\frac{2}{r_0^{\alpha}}\right)\Vert
Q\Vert_1\,,\quad m\geqslant M_0\,.$$
The right part of the last relation does not depend on~$m,$ so we
may put
$$N:=\left(\frac{N_04^{\alpha}}{\varepsilon^{\alpha}_0}+\frac{2}{r_0^{\alpha}}\right)\Vert
Q\Vert_1\,.$$ Lemma~\ref{lem4A} is proved.~$\Box$
\end{proof}

\medskip
\begin{lemma}\label{lem5}
{\sl\, Let $D^{\,\prime}$ be a bounded regular domain in
$X^{\,\prime},$ which is finitely connected on the boundary. Then
any two pairs of (different) points $a\in D^{\,\prime},$
$b\in\overline{D^{\,\prime}}$ and $c\in D^{\,\prime},$
$d\in\overline{D^{\,\prime}}$ may be joined by disjoint paths
$\alpha:[0, 1]\rightarrow \overline{D^{\,\prime}}$ and $\beta:[0,
1]\rightarrow \overline{D^{\,\prime}}$ such that $\alpha(t),
\beta(t)\in D^{\,\prime}$ for any $t\in [0, 1).$ }
\end{lemma}

\medskip
\begin{proof}
By the regularity of $D^{\,\prime},$ there is a quasiconformal
mapping $\varphi$ of $D^{\,\prime}$ onto some domain $D_0\subset
{\Bbb R}^n$ with a locally quasiconformal boundary.  Fix $a\in
D^{\,\prime},$ $b\in\overline{D^{\,\prime}}$ and $c\in
D^{\,\prime},$ $d\in\overline{D^{\,\prime}}.$ Since $D^{\,\prime}$
is finitely connected on the boundary, by \cite[Lemma~10.6]{ABBS}
there are $P_1, P_2\in E_{D^{\,\prime}}$ such that $I(P_1)=b$ and
$I(P_2)=d.$  Arguing as in the proof of Lemma~\ref{lem1}, we may
prove that $\varphi$ satisfies the estimate similar to~(\ref{eq1G}),
in addition, we observe that $\varphi$ has a continuous extension
$\varphi:\overline{D^{\,\prime}}_P\rightarrow \overline{D_0}.$ Let
$\widetilde{a}=\varphi(a),$ $\widetilde{b}=\varphi(P_1),$
$\widetilde{c}=\varphi(c)$ and $\widetilde{d}=\varphi(P_2).$
Moreover, by~\cite[Theorem~4.1]{Sev$_2$} $\varphi^{\,-1}$ has a
continuous extension $\varphi^{\,-1}: \overline{D_0}\rightarrow
\overline{D^{\,\prime}}_P.$ By Proposition~\ref{pr1} the points
$\widetilde{a}\in D_0,$ $\widetilde{b}\in\overline{D_0}$ and
$\widetilde{c}\in D_0,$ $\widetilde{d}\in\overline{D_0}$ may be
joined by disjoint paths $C_1:[0, 1]\rightarrow \overline{D_0}$ and
$C_2:[0, 1]\rightarrow \overline{D_0}$ such that $C_i(t)\in D_0$ for
any $t\in [0, 1)$ and any $i=1,2.$ Let
$\widetilde{C}_i:=\varphi^{\,-1}(C_i|_{[0, 1)}).$ Since
$\varphi^{\,-1}: \overline{D_0}\rightarrow
\overline{D^{\,\prime}}_P$ is continuous,
$\widetilde{C}_1(t)\rightarrow P_1$ and
$\widetilde{C}_2(t)\rightarrow P_2$ as $t\rightarrow 1-0.$ Since
$I(P_1)=b$ and $I(P_2)=d,$ $\widetilde{C}_1(t)\rightarrow b$ and
$\widetilde{C}_2(t)\rightarrow d$ as $t\rightarrow 1-0.$ Putting%
$$\alpha(t)=\begin{cases}\varphi^{\,-1}(C_1(t))\,,&t\in [0,1)\\
\lim\limits_{t\rightarrow 1-0}\varphi^{\,-1}(C_1(t))\,,
&t=1\end{cases}\,,\quad \beta(t)=\begin{cases}\varphi^{\,-1}(C_2(t))\,,&t\in [0,1)\\
\lim\limits_{t\rightarrow 1-0}\varphi^{\,-1}(C_2(t))\,,
&t=1\end{cases}$$
we obtain the desired paths $\alpha$ and $\beta.$~$\Box$
\end{proof}

\section{On equicontinuity of families at inner points }

The presentation of this section is conceptually close to the
publication~\cite{Skv}. The only significant difference is the
presence of weak sphericalization in the space under consideration,
which is present in~\cite{Skv}, but is not necessary when presenting
the material in this article. The proofs of the results in this
publication and this manuscript are very similar, however, for the
sake of completeness, we present them in full below.

\medskip
Let $\beta:[a, c)\rightarrow D$ be a path in $D\subset (X, d).$ In
what follows, we write $\beta(t)\rightarrow
\partial D,$ if for any $\varepsilon>0$ there is $\delta:=\delta(\varepsilon)>0$
such that $d(\beta(t), \partial D)<\varepsilon$ for any $t\in
(\delta, c).$ Here we assume that $\partial D\ne\varnothing.$ The
following statement holds, which in a more special case is proved
in~\cite[Lemma~2.1]{Skv}.

\begin{lemma}\label{lm1}{\sl\,
Let $D$ and $D^{\,\prime}$ be domains with finite Hausdorff
dimensions $\alpha$ and $\alpha^{\,\prime}\geqslant 2$ in spaces
$(X,d,\mu)$ and $(X^{\,\prime},d^{\,\prime}, \mu^{\,\prime}),$
respectively. Assume that $X$ is locally connected and
$\overline{D}, \overline{D^{\,\prime}}$ are compact sets. Let $f$ be
an open discrete mapping of $D$ onto $D^{\,\prime}.$ Let $\beta:[a,
b)\rightarrow D^{\,\prime}$ be a path such that $\beta(t)\rightarrow
\partial D^{\,\prime}$ as $t\rightarrow b-0$ and let
$\alpha:[a,c)\rightarrow D$ be a maximal $f$-lifting of $\beta$
starting at $x\in f^{\,-1}(\beta(a)).$ Then $d(\alpha(t),\partial
D)\rightarrow 0$ as $t\rightarrow c-0.$}
\end{lemma}

\begin{proof}
First of all, we observe that the formulation of Lemma~\ref{lm1} is
correct, namely, $\partial D$ is not empty. Indeed, since
$\beta(t)\rightarrow\partial D^{\,\prime},$ we have that $\partial
D^{\,\prime}\ne \varnothing.$ Then we may find a sequence of points
$y_k=\beta(t_k)\in D^{\,\prime}$ such that $d^{\,\prime}(y_k,
\partial D^{\,\prime})\rightarrow 0$ as $k\rightarrow \infty,$ where
$t_k\in [a, b).$ Since $\overline{D^{\,\prime}}$ is compact, we may
assume that $y_k\rightarrow y_0\in X^{\,\prime}.$ Then $y_0\in
\partial D^{\,\prime},$ because $d^{\,\prime}(y_k, \partial
D^{\,\prime})\rightarrow 0 $ as $k\rightarrow \infty.$ Since $f(D)=D
^{\,\prime},$ there is $x_k\in D$ such that $f(x_k)=y_k.$ Since
$\overline{D}$ is a compact set, we may also assume that the
sequence $x_k$ converges to some point $x_0$ as $k\rightarrow
\infty.$ Observe that, the point $x_0$ cannot be inner for $D,$
because otherwise due to the openness of the mapping $f$ the point
$y_0$ is inner, as well. Then $x_0\in \partial D,$ hence, $\partial
D \ne \varnothing,$ which needed to be proved.

\medskip
Since we have now established that $\partial D \ne \varnothing,$
only two cases are possible: or $d(\alpha (t), \partial
D)\rightarrow 0 $ as $t\rightarrow c-0,$ or $d(\alpha(p_k), \partial
D)\geqslant \delta_0> 0$ as $k\rightarrow\infty$ for some sequence
$p_k\rightarrow c- 0$ and some $\delta_0>0.$ Let us prove
Lemma~\ref{lm1} from the opposite, in other words, assume that there
is a path $\beta:[a, b) \rightarrow D^{\,\prime}$ and its maximal
$f$-lifting $\alpha:[a, c)\rightarrow D,$ $c\in (a, b),$ such that
the condition $d(\alpha (t), \partial D)\geqslant \delta_0> 0$ holds
for some $\delta_0>0$ and $t\rightarrow c-0. $ Since $\overline{D}$
is a compact set, we may consider that $\alpha(p_k)\rightarrow
x_1\in D$ as $k\rightarrow\infty.$ Put
$$D_0=\left\{x\in X:  x=\lim_{k\rightarrow \infty} \alpha(t_k),\quad
t_k \in [a,c), \quad\lim\limits_{k\rightarrow \infty}
t_k=c\right\}.$$
Observe that $c\ne b.$  Indeed, in the contrary case
$f(x_1)=\beta(c)\in D$ that contradicts the choice of the path
$\beta.$

\medskip
Let now $c\ne b.$ Passing to the subsequences if necessary, we may
limit ourselves to monotonic sequences $t_k.$ Let $x \in D_0\cap D,$
then by the continuity of $f$ we obtain that
$f(\alpha(t_k))\rightarrow f(x) $ as $k\rightarrow \infty,$ where
$t_k \in [a, c),$ $t_k \rightarrow c $ as $k\rightarrow \infty. $
However, $f(\alpha(t_k))=\beta(t_k)\rightarrow \beta (c)$ as $k
\rightarrow \infty. $ Then the mapping $f$ is constant on $D_0\cap
D.$ Since the sequence of connected sets $\alpha([t_k,c))$ is
monotone,
$$ D_0=\bigcap\limits_{k=1}^{\infty} \overline{\alpha([t_k,c))}
\ne\varnothing\,.$$ By~\cite[Theorem~5.II.47.5]{Ku} the set $D_0$ is
connected. Let $L_0$ be a connected component of $D_0\cap D$ that
contains the point $x_1.$ If $D_0$ contains at least two points,
then, by the definition of connectedness, $L_0$ has the same
property. Since $f$ is a discrete mapping, and $L_0\subset D,$ the
set $L_0$ is one-point. Hence, $D_0$ is one-point as well. In this
case, the path $\alpha: [a, c)\rightarrow D$ may be extended to the
closed curve  $\alpha: [a, c]\rightarrow D,$ and
$f(\alpha(c))=\beta(c).$ Then, by Lemma~2.1 in \cite{SM}, there
exists one more maximal $f$-lifting $\alpha^{\,\prime}$ of
$\beta_{[c, b)}$ starting at the point $\beta(c).$ Combining the
liftings $\alpha$ and $\alpha^{\,\prime},$ we obtain a new
$f$-lifting of the path $\beta$ which is defined on $[a,
c^{\,\prime}),$ $c^{\,\prime}\in (c, b).$ This contradicts the
maximality of the lifting $\alpha.$ The obtained contradiction
indicates that $d(\alpha(t),\partial D)\rightarrow 0$ as
$t\rightarrow c-0.$ $\Box$ ~\end{proof}

\medskip
Following~\cite[Section~9]{RS}, a space $X$ is
called {\it weakly flat} at the point $x_0\in X,$ if, for any
neighborhood $U$ of $x_0$ and for every $P>0,$ there exists a
neighborhood $V\subset U$ of $x_0$ such that
$$M_{\alpha}(\Gamma(E, F, X))\geqslant P$$
for any continua $E, F\subset X$ with $E\cap
\partial U=\varnothing\ne E\cap
\partial V$ and $F\cap \partial U=\varnothing\ne F\cap \partial V.$
A space $X$ is called {\it weakly flat}, if the indicated property
holds for any $x_0\in X.$ The following result is similar to
Theorem~1.1 in~\cite{Skv}.

\medskip
We say that a space  $(X,d,\mu)$ is {\it upper $\alpha$-regular at a
point} $x_0\in X$ if there is a constant $C> 0$ such that
$$
\mu(B(x_0,r))\leqslant Cr^{\alpha}$$
for the balls $B(x_0,r)$ centered at $x_0\in X$ with all radii
$r<r_0$ for some $r_0>0.$ We will also say that a space  $(X,d,\mu)$
is {\it upper $\alpha$-regular} if the above condition holds at
every point $x_0\in X.$ It follows from~\cite[formulae~2.1 and
2.2]{ABBS} that doubling measure spaces are locally $\alpha$-Ahlfors
regular, moreover, arguing similarly to \cite[section~8.7,
p.~61]{He} we may show that the number $\alpha$ equals to Hausdorff
dimension of $X.$

\medskip
Given $M>0$ and domains $D\subset X, D^{\,\prime}\subset
X^{\,\prime},$ denote by ${\frak S}_M(D, D^{\,\prime})$ a family of
all open discrete and closed mappings $f$ of $D$ onto $D^{\,\prime}$
such that the condition~(\ref{eq2*A}) holds for any $y_0\in
D^{\,\prime}$ for some $Q=Q_f$ and $\Vert
Q_f\Vert_{L^1(D^{\,\prime})}\leqslant M.$

\medskip
\begin{theorem}\label{th4A}{\sl\, Let $D$ and $D^{\,\prime}$ be domains with finite Hausdorff
dimensions $\alpha$ and $\alpha^{\,\prime}\geqslant 2$ in spaces
$(X,d,\mu)$ and $(X^{\,\prime},d^{\,\prime}, \mu^{\,\prime}),$
respectively. Assume that $X$ is locally connected, $X^{\,\prime}$
is complete and supports $\alpha^{\,\prime}$-Poincar\'{e}
inequality, and that the measures $\mu$ and $µ^{\,\prime}$ are
doubling. Suppose that $\overline{D}$ and $\overline{D^{\,\prime}}$
are compact sets, and $D$ is weakly flat as a metric space. Let
$D^{\,\prime}\subset X^{\,\prime}$ be a regular domain which is
finitely connected at the boundary. Then the family ${\frak S}_M(D,
D^{\,\prime})$ is equicontinuous in $D.$
 }
\end{theorem}

\medskip
\begin{proof}
Let us prove Theorem~\ref{th4A} by contradiction. Suppose that a
family ${\frak S}_M(D, D^{\,\prime})$ is not equicontinuous at some
point $x_0\in D,$ in other words, there are $x_0\in D$ and
$\varepsilon_0> 0$ with the following condition: for any $m\in {\Bbb
N}$ there is $x_m\in D,$ and a mapping $f_m \in {\frak S}_M(D,
D^{\,\prime})$ such that $d(x_0, x_m)<\frac{1}{m}$ and
\begin{equation}\label{eq13***}
d^{\,\prime}(f_m(x_0), f_m(x_m))\geqslant \varepsilon_0\,.
\end{equation}
Since $\overline{D^{\prime}}$ is a compact set, we may consider that
the sequences $f_m(x_0)$ and $f_m(x_m)$ converge to $a_1$ and $a_2
\in \overline {D^{\prime}}$ as $m\rightarrow \infty$. Then,
by~(\ref{eq13***}), by the triangle inequality and by the continuity
of the metric $d^{\,\prime}$ we obtain that $d^{\,\prime}(a_1,
a_2)\geqslant \varepsilon_0.$

\medskip
Let us show that there are at least two boundary points  $w_1, w_2
\in \partial D^{\prime},$ $w_1\ne w_2.$  Since, by condition,
$D^{\,\prime}$ is a regular domain, it can be mapped onto some
domain with a locally quasiconformal boundary by (some)
quasiconformal mapping $h:D^{\,\prime}\rightarrow D_0.$ By the
definition of a domain with a locally quasiconformal boundary, such
a domain has an infinite number of boundary points. By the remarks
made before the statement of Theorem~\ref{th4A}, the space
$X^{\,\prime}$ is locally Ahlfors regular with exponent
$\alpha^{\,\prime}$ equal to the Hausdorff dimension of the space
$X^{\,\prime}.$ By Proposition~\ref{pr2A} $D_0$ has a weakly flat
boundary. Now, by~\cite[Theorem~5]{Sm} $h$ may be extended to a
homeomorphism $\overline{h}:\overline{D^{\,\prime}}\rightarrow
\overline{D_0}.$ Since the domain $D_0$ has infinitely many boundary
points, it follows that the domain $D ^{\,\prime}$ is also the same
as required to prove.

\medskip
Define the paths $\gamma_1$ and $\gamma_2 $ as follows. If both
points $a_1$ and $a_2$ are boundary, then we denote by $\gamma_i$
the degenerate path, the image of which coincides with the
corresponding point $a_i,$ $i=1,2.$ If only one of the points $a_1$
or $a_2$ is boundary, for example, the point $a_1,$ then again
denote by $\gamma_1$ the corresponding degenerate path, the image of
which coincides with $a_1;$ let, for example, $a_1\ne w_2,$ then
join the point $a_2$ with the point $w_2$ by a path $\gamma_2,$
which lies in $D^{\,\prime}$ completely, except endpoint $w_2$ (this
is possible due to the Lemma~\ref{lem5}).

Finally, if both points $a_1$ and $a_2$ are inner, join the points
$a_1$ and $w_1,$ $a_2$ and $w_2$ with non-intersecting  paths
$\gamma_1$ and $\gamma_2$ in $D^{\,\prime},$ which is possible due
to the Lemma~\ref{lem5} (see Figure~\ref{fig1}). Note that
$|\gamma_i|, i= 1,2$ are compact sets in $X^{\prime}$ as continuous
images of the corresponding segments in the metric space, so there
is $l_0> 0$ such that $d^{\,\prime}(|\gamma_1|, |\gamma_2|)=l_0>0.$
\begin{figure}[h]\label{fig1}
\center{\includegraphics[scale=0.7]{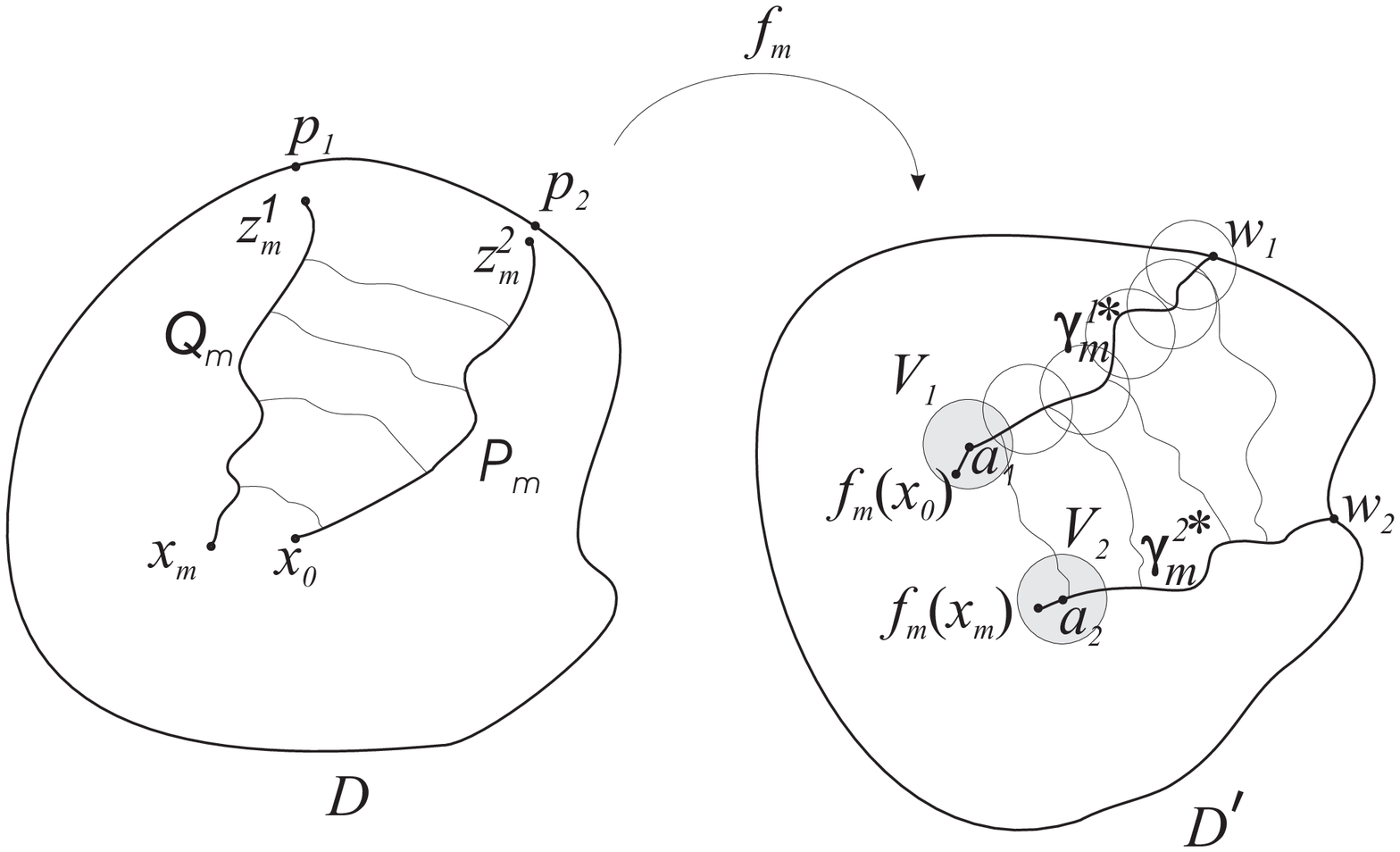}} \caption{To the
proof of Theorem~\ref{th4A}}
\end{figure}
Recall that, the space $X^{\,\prime}$ is complete and supports a
$\alpha^{\,\prime}$-Poincar\'{e} inequality (see \cite{ABBS}). Thus,
$X^{\,\prime}$ is locally connected (see \cite[Section~2]{ABBS}),
and proper (see \cite[Proposition~3.1]{BB}). Therefore,
$X^{\,\prime}$ is locally path connected by the
Mazurkiewicz–Moore–Menger theorem (see~\cite[Theorem~1, Ch.~6, $\S$
50, item II]{Ku}). Then there are path connected neighborhoods $V_i$
of the points $a_i,$ $i=1,2 $ such that $V_1\cap V_2=\varnothing.$
We may assume that $V_1\subset B(a_1, \frac{l_0}{4}).$ Join the
points $a_1$ and $a_2$ with the points $f_m(x_0)$ and $f_m(x_m)$ by
the paths $\alpha_m$ and $\beta_m$ within the neighborhoods $V_1$
and $V_2,$ respectively. Now let us consider new paths, namely, let
$\gamma_m^{1*}:[0,1]\rightarrow X^{\,\prime}$ be a path, which is a
union of paths $\gamma_1$ and $\alpha_m,$ and let
$\gamma_m^{2*}:[0,1]\rightarrow X^{\,\prime} $ be a path that is a
union of the paths $\gamma_2$ and $\beta_m,$ where
$\gamma_m^{1*}(0)=f_m(x_0)$ and $\gamma_m^{2*}(0)=f_m(x_m).$ Let
Put $$\gamma_m^1=\gamma_m^{1*}|_{[0, 1)}\,,
\gamma_m^2=\gamma_m^{2*}|_{[0, 1)}.$$
By~\cite[Lemma~2.1]{SM} there are maximal $f_m$-liftings
$\alpha_m^*$ and $\beta_m^*$ of the paths $\gamma_m^1$ and
$\gamma_m^2$ starting at the points $x_0$ and $x_m,$ respectively.

\medskip
By Lemma~\ref{lm1}, $d(\alpha_m^*(t), \partial D)\rightarrow 0$ as
$t\rightarrow c_1-0$ and $d(\beta_m^*(t), \partial D)\rightarrow 0$
as $t\rightarrow c_2-0.$ Thus, we may find points $z_m^1$ and
$z_m^2$ on the paths $\alpha_m^*$ and $\beta_m^*$ such that
$d(z_m^1,
\partial D)<\frac{1}{m}$ and $d(z_m^2,\partial D)<\frac{1}{m}.$
Since $\overline{D}$ is a compact set, we may consider that $z_m^1$
and $z_m^2$ converge to points $p_1$ and $p_2$ as
$m\rightarrow\infty,$ respectively. Let $P_m$ be a part of the locus
of the path $\alpha_m^*$ in $X,$ which is located between the points
$x_0$ and $z_m^1,$ and let $Q_m$ be a part of the locus of the path
$\beta_m^*$ in $X,$ located between the points $x_m$ and $z_m^2.$
Consider the family
$$\Gamma_m^*:=\Gamma(P_m,Q_m,D)\,.$$
By the definition of a maximal $f_m$-liftings starting at the points
$x_0$ and $x_m$ we obtain that
\begin{equation}\label{vkluch1}
f_m(P_m) \subset |\gamma_m^{1}|, \,\, f_m(Q_m) \subset
|\gamma_m^{2}|.
\end{equation}
Consider the coverage of $A_0:=\bigcup\limits_{y\in |\gamma_1|}B(y,
l_0/4)$ of set $|\gamma_1|.$ Since $|\gamma_1|$ is compact set, by
the Heine-Borel-Lebesgue lemma we may choose a finite number of
indices $1\leqslant N_0 < \infty$ and the corresponding points
$z_1,\ldots, z_{N_0}\in |\gamma_1|$ so that $|\gamma_1|\subset
B_0:=\bigcup\limits_{i=1}^{N_0}B(z_i, l_0/4).$ Without limiting the
generality, we may assume that all points $z_i$ belong
$D^{\,\prime},$ because otherwise by the triangle inequality we may
replace the ball $B(y, l_0/4)$ with a ball $B(y^*, l^*_0/4)$
centered at some point $y^*\in D^{\prime},$ where $l^*_0$ is some
positive number that can be chosen so that $ l_0 / 4 <l ^
* _ 0/4 <l_0 / 2. $ Also, since $V_1\subset B (a_1, \frac {l_0}{4}),$ we may
assume that $V_1 \subset B(z_{i_0}, \frac{l_0}{4})$ for some
$z_{i_0} \in D^{\,\prime}$ and $1\leqslant i_0 \leqslant N_0.$

Let $\Gamma_m^{\,\prime}$ be a family of all paths joining
$f_m(P_m)$ and $f_m(Q_m)$ in $D^{\,\prime},$ and let $\Gamma_m$ be a
family of all paths joining $|\gamma_m^{1}|$ and $|\gamma_m^{2}|$ in
$D^{\prime}.$ Then, using (\ref{vkluch1}) we obtain that
\begin{equation}\label{eq10CC}
\Gamma_m^{\,\prime}\subset
\Gamma_m=\bigcup\limits_{i=1}^{N_0}\Gamma_{mi}\,,
\end{equation}
where $\Gamma_{mi}$ is a family of all paths $\gamma:[0,
1]\rightarrow D^{\,\prime}$ such that $\gamma(0)\in B(z_i,
l_0/4)\cap |\gamma_m^{1}|$ and $\gamma(1)\in |\gamma_m^{2}|$ as
$1\leqslant i\leqslant N_0.$ Due to~\cite[Theorem~1.I.5, \S 46]{Ku},
we may show that
\begin{equation}\label{eq11C}
\Gamma_{mi}>\Gamma(S(z_i, l_0/4), S(z_i, l_0/2), A(z_i, l_0/4,
l_0/2))\,.
\end{equation}
Let $\gamma \in \Gamma_m^*,$ then $\gamma:[0,1]\rightarrow D,$
$\gamma(0) \in P_m,$ $\gamma(1) \in Q_m.$ In particular,
$f_m(\gamma(0)) \in f_m(P_m),$ $f_m(\gamma(1)) \in f_m(Q_m).$ Due to
the relation~(\ref{eq10CC}), we obtain that $f_m(\gamma) \in
\Gamma_{mi}$ for some $1\leqslant i \leqslant N_0.$ Now, it follows
from~(\ref{eq11C}) that $f_m(\gamma)$ has a subpath
$\triangle:(t_1,t_2)\rightarrow D^{\,\prime}$ such that $\triangle
\in \Gamma(S(z_i, \frac{l_0}{4}), S(z_i, \frac{l_0}{2}), A(z_i,
\frac{l_0}{4}, \frac{l_0}{2})).$ Thus, by the definition
$\gamma_1=\gamma|_{[t_1, t_2]}$ and $f_m(\gamma_1)=\triangle \in
\Gamma(S(z_i, \frac{l_0}{4}), S(z_i, \frac{l_0}{2}), A(z_i,
\frac{l_0}{4}, \frac{l_0}{2})).$ In other words, $\gamma>\gamma_1,$
where $\gamma_1 \in \Gamma_{f_m}(z_i, \frac{l_0}{4},
\frac{l_0}{2}).$ Thus
%
$$\Gamma_m^*>\bigcup\limits_{i=1}^{N_0}
\Gamma_{f_m}\left(z_i,\frac{l_0}{4},\frac{l_0}{2}\right)\,.$$
%
Set
$$\eta(t)= \left\{
\begin{array}{rr}
4/l_0, & t\in [l_0/4, l_0/2],\\
0,  &  t\not\in [l_0/4, l_0/2]\,.
\end{array}
\right. $$
Since $f_m$ satisfy~(\ref{eq2*A}) in $D^{\prime},$ we obtain that
$$M_{\alpha}(\Gamma_m^*)\leqslant
\sum\limits_{i=1}^{N_0}M_{\alpha}\left(\Gamma_{f_m}\left(z_i,\frac{l_0}{4},\frac{l_0}{2}\right)\right)
\leqslant
$$
\begin{equation}\label{eq14***}
\leqslant 4^{\alpha^{\prime}}(N_0/l_0^{\alpha^{\prime}})\cdot\Vert
Q_{f_m}\Vert_1\leqslant
4^{\alpha^{\prime}}(N_0/l_0^{\alpha^{\prime}})M<\infty\,.
\end{equation}
Here $\Vert Q_{f_m}\Vert_1$ denotes the norm of the function
$Q_{f_m}$ in $L^1(D^{\,\prime}).$

\medskip
On the other hand, for sufficiently large $m$ we obtain that
$$d(P_m)\geqslant d(z_m^1,x_0)\geqslant d(x_0,p_1)-d^{\,\prime}(z_m^1,p_1)\geqslant
\frac{1}{2}d(x_0,p_1)\,,$$
$$d(Q_m)\geqslant d(z_m^2,x_m)\geqslant d(z_m^2,x_0)-d(x_m,x_0)\geqslant $$
$$\geqslant d(p_2,x_0)-d(z_m^2,p_2)-d(x_m,x_0) \geqslant
\frac{1}{2}d(x_0,p_2)\,.$$
Set
$$U=B\left(x_0,\frac{R_0}{2}\right)=\left\{x\in X: d(x, x_0)<\frac{R_0}{2}\right\}\,,$$
where
$$R_0=\min\{d(x_0, p_1), d(x_0, p_2), d(x_0,\partial D)\}\,.$$
Since $x_m\rightarrow x_0$ as $m\rightarrow \infty,$ in addition,
$d(P_m)\geqslant \delta_1$ and $d(Q_m)\geqslant \delta_2$ for any
$m\geqslant m_0$ and some $m_0 \in {\Bbb N},$ due
to~\cite[Theorem~1.I.5, \S 46]{Ku} there is $m_1\in {\Bbb N}$ such
that $P_m\cap
\partial U \ne \varnothing \ne Q_m\cap \partial U$ for $m\geqslant m_1.$
Since $D$ is a weakly flat space, for any $P>0$ there is a
neighborhood $V\subset U$ of the point $x_0$ such that
\begin{equation}\label{eq4}
M_{\alpha}(\Gamma(E, F, D))>P
\end{equation}
for any continua $E, F\subset D$ such that $E\cap
\partial U \ne \varnothing \ne F\cap \partial U$ and $E\cap
\partial V \ne \varnothing \ne F\cap \partial V.$ Since
$x_m\rightarrow x_0$ as $m\rightarrow \infty,$ in addition,
$d(P_m)\geqslant \delta_1$ and $d(Q_m)\geqslant \delta_2$ for any
$m\geqslant m_0$ and some $m_0 \in {\Bbb N},$ due
to~\cite[Theorem~1.I.5, \S 46]{Ku} there is $m_2\in {\Bbb N},$
$m_2>m_1,$ such that $P_m\cap
\partial V \ne \varnothing \ne Q_m\cap \partial V.$ Then by~(\ref{eq4}) it follows that
\begin{equation}\label{eq5A}
M_{\alpha}(\Gamma_m(P_m, Q_m, D))>P\,,\quad m>m_2\,.
\end{equation}
The relation~(\ref{eq5A}) contradicts with~(\ref{eq14***}) since a
number $P$ my be chosen greater than
$4^{\alpha^{\prime}}(N_0/l_0^{\alpha^{\prime}})M.$ The obtained
contradiction proves the theorem.~$\Box$
\end{proof}

\section{Proof of Theorem~\ref{th2}}

The possibility of continuous extension of the mapping~$f\in {\frak
S}_M(D, D^{\,\prime})$ to $\partial D$ follows by
Theorem~\ref{th1A}. Since ${\frak S}_{\delta, A, M}(D,
D^{\,\prime})\subset {\frak S}_M(D, D^{\,\prime}),$ the
equicontinuity of~${\frak S}_{\delta, A, M}(D, D^{\,\prime})$ at
inner points of $D$ follows by Theorem~\ref{th4A}.

\medskip
Let us to show the equicontinuity of~${\frak S}_{\delta, A, М
}(\overline{D}, \overline{D^{\,\prime}})$ on $\partial D.$ Assume
the contrary.

Now, there is a point $z_0\in
\partial D,$ a number~$\varepsilon_0>0,$ a sequence $z_m\in
\overline{D}$ and a mapping $\overline{f}_m\in {\frak S}_{\delta, A,
M}(\overline{D}, \overline{D^{\,\prime}})$ such that $z_m\rightarrow
z_0$ as $m\rightarrow\infty$ and
\begin{equation}\label{eq12A}
m_P(\overline{f}_m(z_m),
\overline{f}_m(z_0))\geqslant\varepsilon_0,\quad m=1,2,\ldots ,
\end{equation}
where $m_P$ is some of possible metrics
in~$\overline{D^{\,\prime}}_P$ defined in Proposition~\ref{pr2}.
Since $f_m=\overline{f}_m|_{D}$ has a continuous extension to
$\overline{D}_P,$ we may assume that $z_m\in D$ and, in addition,
there is one more sequence $z^{\,\prime}_m\in D,$
$z^{\,\prime}_m\rightarrow z_0$ as $m\rightarrow\infty$ such that
$m_P(f_m(z^{\,\prime}_m), \overline{f}_m(z_0))\rightarrow 0$ as
$m\rightarrow\infty.$ In this case, it follows from~(\ref{eq12A})
that
$$
m_P(f_m(z_m), f_m(z^{\,\prime}_m))\geqslant\varepsilon_0/2,\quad
m\geqslant m_0\,.
$$
By~\cite[Theorem~10.10]{ABBS}, $\left(\overline{D^{\,\prime}}_P,
m_P\right)$ is a compact metric space. Thus, we may assume
that~$f_m(z_m)$ and $f_m(z_m^{\,\prime})$ converge to some $P_1,
P_2\in \overline{D^{\,\prime}}_P,$ $P_1\ne P_2,$ as
$m\rightarrow\infty.$ Let $d_m$ and $g_m$ be sequences of decreasing
domains corresponding to prime ends $P_1$ and $P_2,$ respectively.
Put $x_0, y_0\in A$ such that $x_0\ne y_0$ and $x_0\ne P_1\ne y_0,$
where the continuum~$A\subset D^{\,\prime}$ is taken from the
conditions of Theorem~\ref{th2}. Without loss of generality, we may
assume that $d_1\cap g_1=\varnothing$ and $x_0, y_0\not\in d_1\cup
g_1.$

\medskip
By Lemmas~\ref{lem1} and~\ref{lem4A}, we may find disjoint paths
$\gamma_{1,m}:[0, 1]\rightarrow D^{\,\prime}$ and $\gamma_{2,m}:[0,
1]\rightarrow D^{\,\prime}$ and a number $N> 0$ such that
$\gamma_{1, m}(0)=x_0,$ $\gamma_{1, m}(1)=f_m(z_m),$ $\gamma_{2,
m}(0)=y_0,$ $\gamma_{2, m}(0)=f_m(z^{\,\prime}_m),$ wherein
\begin{equation}\label{eq15}
M_{\alpha}(\Gamma_m)\leqslant N\,,\quad m\geqslant M_0\,,
\end{equation}
where $\Gamma_m$ consists of those and only those paths $\gamma$ в
$D$ for which $f_m(\gamma)\in\Gamma(|\gamma_{1, m}|, |\gamma_{2,
m}|, D^{\,\prime})$ (see Figure~\ref{fig6}).
\begin{figure}
\centering\includegraphics[width=400pt]{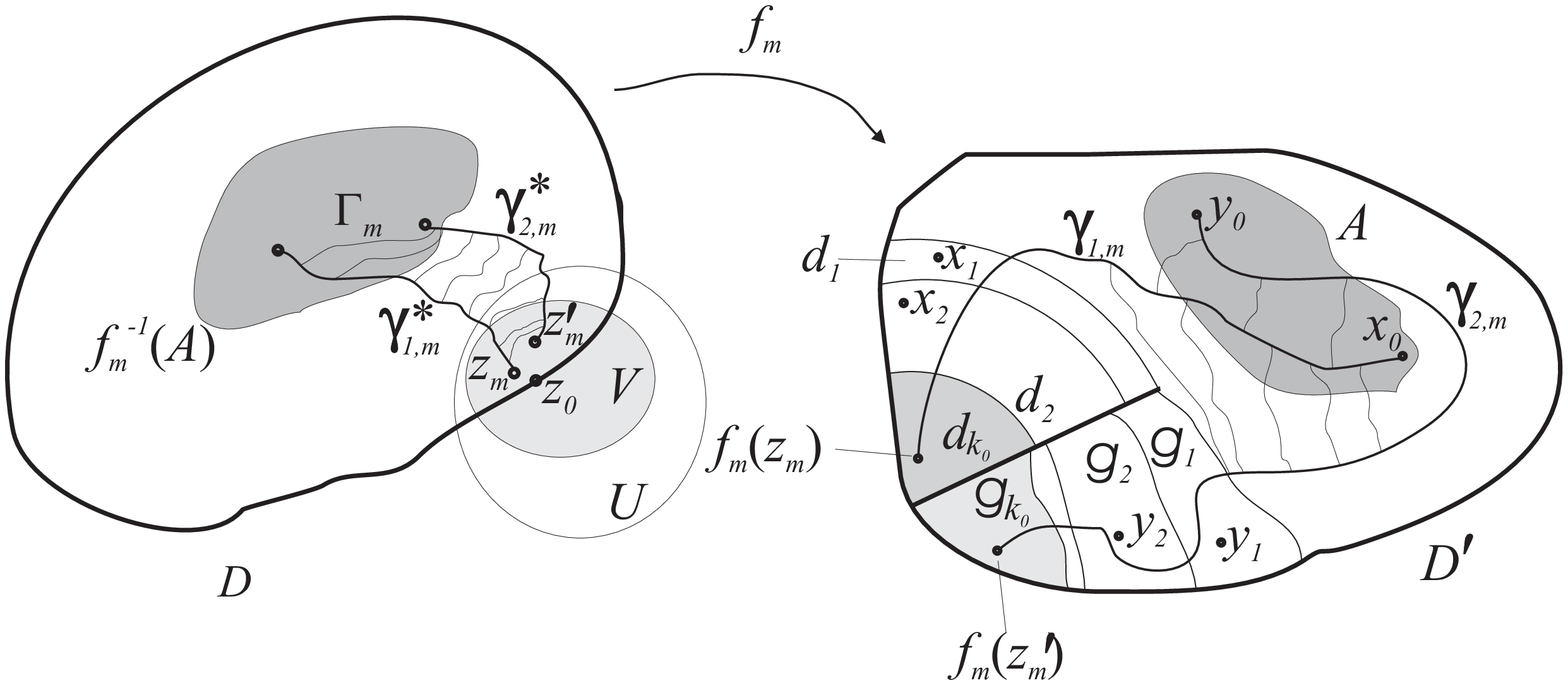} \caption{To
proof of Theorem~\ref{th2}.}\label{fig6}
\end{figure}
On the other hand, let~$\gamma^*_{1,m}$ and~$\gamma^*_{2,m}$ be the
total $f_m$-lifting of the paths~$\gamma_{1,m}$ and~$\gamma_{2,m}$
starting at the points~$z_m$ and $z^{\,\prime}_m,$ respectively
(such lifting exist by Lemma~\ref{lem5}). Now, $\gamma^*_{1,m}(1)\in
f^{\,-1}_m(A)$ and $\gamma^*_{2,m}(1)\in f^{\,-1}_m(A).$ Since by
the condition $d(f^{\,-1}_{m}(A),
\partial D)>\delta>0,$ $m=1,2,\ldots \,,$ we obtain that
$$d(|\gamma^*_{1, m}|)\geqslant d(z_m, \gamma^*_{1,m}(1)) \geqslant
(1/2)\cdot d(f^{\,-1}_m(A), \partial D)>\delta/2\,,$$
\begin{equation}\label{eq14A}
d(|\gamma^*_{2, m}|)\geqslant d(z^{\,\prime}_m, \gamma^*_{2,m}(1))
\geqslant (1/2)\cdot d(f^{\,-1}_m(A), \partial D)>\delta/2
\end{equation}
for sufficiently large $m\in {\Bbb N}.$
Choose the ball $U:=B(z_0, r_0)=\{z\in X: d(z, z_0)<r_0\},$ where
$r_0>0$ and $r_0<\delta/4.$ Observe that $|\gamma^*_{1, m}|\cap
U\ne\varnothing\ne |\gamma^*_{1, m}|\cap (D\setminus U)$ for
sufficiently large $m\in{\Bbb N},$ because $d(f_m(|\gamma_{1,
m}|))\geqslant \delta/2$ and $z_m\in|\gamma^*_{1, m}|,$
$z_m\rightarrow z_0$ as $m\rightarrow\infty.$ Arguing similarly, we
may conclude that~$|\gamma^*_{2, m}|\cap U\ne\varnothing\ne
|\gamma^*_{2, m}|\cap (D\setminus U).$ Since $|\gamma^*_{1, m}|$ and
$|\gamma^*_{2, m}|$ are continua, by~\cite[Theorem~1.I.5.46]{Ku}
\begin{equation}\label{eq8AA}
|\gamma^*_{1, m}|\cap \partial U\ne\varnothing, \quad |\gamma^*_{2,
m}|\cap
\partial U\ne\varnothing\,.
\end{equation}
Put $P:=N>0,$ where $N$ is a number from the relation~(\ref{eq15}).
Since $D$ has a weakly flat boundary, we may find a neighborhood
$V\subset U$ of $z_0$ such that
\begin{equation}\label{eq9AA}
M_{\alpha}(\Gamma(E, F, D))>N
\end{equation}
for any continua $E, F\subset D$ with $E\cap
\partial U\ne\varnothing\ne E\cap \partial V$ and $F\cap \partial
U\ne\varnothing\ne F\cap \partial V.$ Observe that
\begin{equation}\label{eq10AB}
|\gamma^*_{1, m}|\cap \partial V\ne\varnothing, \quad |\gamma^*_{2,
m}|\cap
\partial V\ne\varnothing\,.\end{equation}
for sufficiently large $m\in {\Bbb N}.$ Indeed, $z_m\in
|\gamma^*_{1, m}|$ and $z^{\,\prime}_m\in |\gamma^*_{2, m}|,$ where
$z_m, z^{\,\prime}_m\rightarrow z_0\in V$ as $m\rightarrow\infty.$
Thus, $|\gamma^*_{1, m}|\cap V\ne\varnothing\ne |\gamma^*_{2,
m}|\cap V$ for sufficiently large $m\in {\Bbb N}.$ Besides that,
$d(V)\leqslant d(U)=2r_0<\delta/2$ and $|\gamma^*_{1, m}|\cap
(D\setminus V)\ne\varnothing$ because $d(|\gamma^*_{1,
m}|)>\delta/2$ by~(\ref{eq14A}). Then $|\gamma^*_{1, m}|\cap\partial
V\ne\varnothing$ (see~\cite[Theorem~1.I.5.46]{Ku}). Similarly,
$d(V)\leqslant d(U)=2r_0<\delta/2.$ Now, since by~(\ref{eq14A})
$d(|\gamma^*_{2, m}|)>\delta/2,$ we obtain that $|\gamma^*_{2,
m}|\cap (D\setminus V)\ne\varnothing.$
By~\cite[Theorem~1.I.5.46]{Ku}, $|\gamma^*_{1, m}|\cap\partial
V\ne\varnothing.$ Thus, (\ref{eq10AB}) is proved. By~(\ref{eq9AA}),
(\ref{eq8AA}) and (\ref{eq10AB}), we obtain that
\begin{equation}\label{eq6a}
M_{\alpha}(\Gamma(|\gamma^*_{1, m}|, |\gamma^*_{2, m}|, D))>N\,.
\end{equation}
The inequality~(\ref{eq6a}) contradicts with~(\ref{eq15}), since
$\Gamma(|\gamma^*_{1, m}|, |\gamma^*_{2, m}|, D)\subset \Gamma_m$
and consequently
$$M_{\alpha}(\Gamma(|\gamma^*_{1, m}|, |\gamma^*_{2, m}|, D))
\leqslant M_{\alpha}(\Gamma_m)\leqslant N\,,\qquad m\geqslant
M_0\,.$$
The obtained contradiction indicates the incorrectness of the
assumption in~(\ref{eq12A}). Theorem~\ref{th2} is proved.~$\Box$

\section{Lemma on the continuum}

One of the versions of the following statement is established
in~\cite[item~v, Lemma~2]{SevSkv$_1$} for homeomorphisms and
''good'' boundaries, see also~\cite[Lemma~4.1]{SevSkv$_2$}. Let us
also point out the case relating to mappings with branching and good
boundaries, see~\cite[Lemma~6.1]{SSD}, as well as the case of bad
boundaries and homeomorphisms, see~\cite[Lemma~2.13]{ISS}. The
statement below seems to refer to the most general situation when a
mapping acts between domains of metric spaces. To the indicated
degree of generality, this statement is proved for the first time.

\medskip
\begin{lemma}\label{lem3}
{\sl\, Let $D$ and $D^{\,\prime}$ be domains with finite Hausdorff
dimensions $\alpha$ and $\alpha^{\,\prime}\geqslant 2$ in spaces
$(X,d,\mu)$ and $(X^{\,\prime},d^{\,\prime}, \mu^{\,\prime}),$
respectively.  Assume that $X$ is locally connected, $\overline{D}$
is compact, $X^{\,\prime}$ is complete and supports
$\alpha^{\,\prime}$-Poincar\'{e} inequality, and that the measures
$\mu$ and $µ^{\,\prime}$ are doubling. Assume that $D$ has a weakly
flat boundary, none of the components of which degenerates into a
point, $\overline{D}$ is a compact set and $D^{\,\prime}$ be a
bounded regular domain in $X^{\,\prime},$ which is finitely
connected on the boundary. Let $A$ be a non-degenerate continuum in
$D^{\,\prime}$ and $\delta>0.$ Assume that $f_m$ is a sequence of
open discrete and closed mappings of $D$ onto $D^{\,\prime}$
satisfying the following condition: for any $m=1,2,\ldots$ there is
a continuum $A_m\subset D,$ $m=1,2,\ldots ,$ such that $f_m(A_m)=A$
and $d(A_m)\geqslant \delta>0.$ Let $\partial D\ne \varnothing.$ If
there is $0<M_1<\infty$ such that $f_m$ satisfies~(\ref{eq2*A}) at
any $y_0\in D^{\,\prime}$ and $m=1,2,\ldots $ with some $Q=Q_m(y)$
for which $\Vert Q_m\Vert_{L^1(D^{\,\prime})}\leqslant M_1,$ then
there exists $\delta_1>0$ such that
$$d(A_m,
\partial D)>\delta_1>0\quad \forall\,\, m\in {\Bbb
N}\,.$$
}
\end{lemma}

\begin{proof}
Since $\partial D\ne \varnothing,$ $d(A_m,
\partial D)$ is well-defined. Let us prove this statement by contradiction. Suppose that the
conclusion of the lemma is not true. Then for each $k\in{\Bbb N}$
there is a number $m=m_k$ such that $d(A_{m_k},
\partial D)<1/k.$ We may
assume that the sequence $m_k$ is increasing by $k.$ Since $A_{m_k}$
is compact, there are $x_k\in A_{m_k}$ і $y_k\in
\partial D$ such that $d(A_{m_k},
\partial D)=d(x_k, y_k)<1/k$ (see Figure~\ref{fig3A}).
\begin{figure}[h]
\centerline{\includegraphics[scale=0.6]{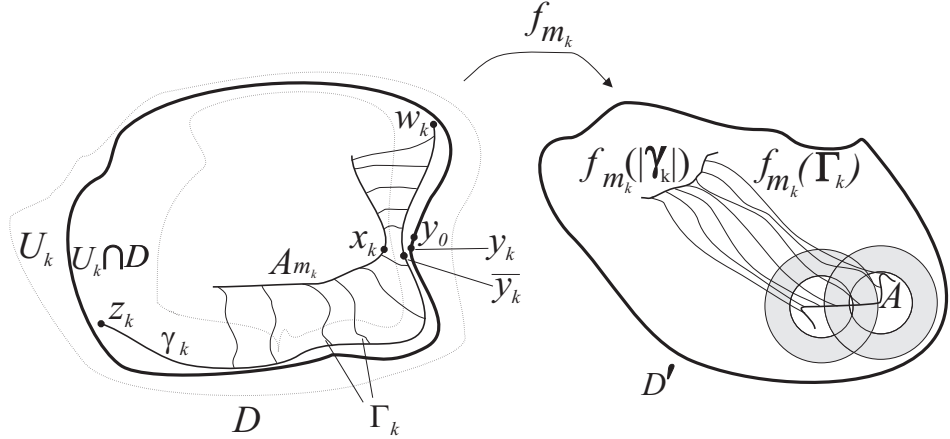}} \caption{To
proof of Lemma~\ref{lem3}}\label{fig3A}
\end{figure}
Since by the assumption $\overline{D}$ is a compact set, $\partial
D$ is a compact set, as well. Then may consider that $y_k\rightarrow
y_0\in
\partial D$ as $k\rightarrow \infty.$ Now we also have that
%
$x_k\rightarrow y_0\in \partial D$ as $k\rightarrow \infty.$
%
Let $K_0$ be a component of $\partial D$ containing $y_0.$
Obviously, $K_0$ is a continuum in $X.$ Since $\partial D$ is weakly
flat, by Theorem~\ref{th1A} $f_{m_k}$ has a continuous extension
$\overline{f}_{m_k}:\overline{D}\rightarrow
\overline{D^{\,\prime}}_P.$ It means that, for any $k=1,2,\ldots,$
any $\varepsilon>0$ and any $x_0\in \overline{D}$ there is
$\delta^{\,*}=\delta^{\,*}_k(\varepsilon, x_0)>0$ such that
\begin{equation}\label{eq1M}
m_P(f_{m_k}(x),\overline{f}_{m_k}(x_0))<\varepsilon
\end{equation}
for any $x\in D$ such that $d(x, x_0)<\delta^{\,*}.$ Let, as usual,
$I(\overline{f}_{m_k}(x_0))$ denotes the impression of the prime end
$\overline{f}_{m_k}(x_0)$ and let
$\overline{f}_{m_k}(x_0)=[E_{k_n}],$ $n=1,2,\ldots ,$ where
$E_{k_n}$ is a sequence of acceptable sets for the prime end
$\overline{f}_{m_k}(x_0).$  Using~(\ref{eq1M}) and the relation
$$E_{k_n}\subset D^{\,\prime}\cap B(I(\overline{f}_{m_k}(x_0)), r)\,, \quad r>0,\quad n=n(r, k)\in
{\Bbb N}\,,$$
see~\cite[relation~(2.2), proof of Lemma~2.1]{Sev$_1$}, we conclude
that $d^{\,\prime}(f_{m_k}(x), I(\overline{f}_{m_k}(x_0))\rightarrow
0$ as $x\rightarrow x_0.$ Now we conclude that there is (possibly,
some another) $\delta=\delta_k(\varepsilon, x_0)>0$ such that
\begin{equation}\label{eq3BA}
d^{\,\prime}(f_{m_k}(x),I(\overline{f}_{m_k}(x_0)))<\varepsilon
\end{equation}
for any $x\in D$ such that $d(x, x_0)<\delta.$ If follows
from~(\ref{eq3BA}) that any mapping $f_{m_k}$ is also continuous in
$\overline{D}$ as a mapping from $\overline{D}$ in
$\overline{D^{\,\prime}}.$ Moreover, the mapping
$\overline{f}_{m_k}$ is uniformly continuous in $\overline{D}$ for
any fixed $k,$ because $\overline{f}_{m_k}$ is continuous on the
compact set $\overline{D}.$ Now, for any $\varepsilon>0$ there is
$\delta_k=\delta_k(\varepsilon)<1/k$ such that
\begin{equation}\label{eq3BB}
d^{\,\prime}(f_{m_k}(x), I(\overline{f}_{m_k}(x_0)))<\varepsilon
\end{equation}
for all $x\in D$ and $x_0\in \overline{D}$ such that $d(x,
x_0)<\delta_k<1/k.$
Let $\varepsilon>0$ be some number such that
\begin{equation}\label{eq5D}
\varepsilon<(1/2)\cdot d^{\,\prime}(\partial D^{\,\prime}, A)\,,
\end{equation}
where $A$ is a continuum from the conditions of lemma and
$g:D_0\rightarrow D^{\,\prime}$ is a quasiconformal mapping of $D_0$
onto $D,$ while $D$ is a domain with a quasiconformal boundary
corresponding to the definition of the metric $\rho.$
Given $k\in {\Bbb N},$ we set
$$B_k:=\bigcup\limits_{x_0\in K_0}B(x_0, \delta_k)\,,\quad k\in {\Bbb
N}\,.$$
Since $B_k$ is a neighborhood of a continuum $K_0,$ by~\cite[Lemma~2
(iii)]{SevSkv$_1$} there is a neighborhood $U_k$ of $K_0$ such that
$U_k\subset B_k$ and $U_k\cap D$ is connected. We may consider that
$U_k$ is open, so that $U_k\cap D$ is linearly path connected
(see~\cite[Proposition~13.1]{MRSY}). Let $d(K_0)=m_0.$ Then we may
find $z_0, w_0\in K_0$ such that $d(K_0)=d(z_0, w_0)=m_0.$ Thus,
there are sequences $\overline{y_k}\in U_k\cap D,$ $z_k\in U_k\cap
D$ and $w_k\in U_k\cap D$ such that $z_k\rightarrow z_0,$
$\overline{y_k}\rightarrow y_0$ and $w_k\rightarrow w_0$ as
$k\rightarrow\infty.$ We may consider that
\begin{equation}\label{eq2B}
d(z_k, w_k)>m_0/2\quad \forall\,\, k\in {\Bbb N}\,.
\end{equation}
Since the set $U_k\cap D$ is linearly connected, we may joint the
points $z_k,$ $\overline{y_k}$ and $w_k$ using some path
$\gamma_k\in U_k\cap D.$ As usually, we denote by $|\gamma_k|$ the
locus of the path $\gamma_k$ in $D.$ Then $f_{m_k}(|\gamma_k|)$ is a
compact set in $D^{\,\prime}.$ If $x\in|\gamma_k|,$ then we may find
$x_0\in K_0$ such that $x\in B(x_0, \delta_k).$ Fix $\omega\in
A\subset D.$ Since $x\in|\gamma_k|$ and, in addition, $x$ is an
inner point of $D,$ we may use the notation $f_{m_k}(x)$ instead
$\overline{f}_{m_k}(x).$
By~(\ref{eq3BB}), (\ref{eq5D}) and by the triangle inequality, we
obtain that
\begin{equation}\label{eq4C}
d^{\,\prime}(f_{m_k}(x),\omega)\geqslant d^{\,\prime}(\omega,
I(\overline{f}_{m_k}(x_0)))-d^{\,\prime}(I(\overline{f}_{m_k}(x_0)),
f_{m_k}(x))\geqslant (1/2)\cdot d^{\,\prime}(\partial D^{\,\prime},
A)>\varepsilon
\end{equation}
for $k\in {\Bbb N}.$ Passing to $\inf$ in~(\ref{eq4C}) over all
$x\in |\gamma_k|$ and $\omega\in A,$ we obtain that
\begin{equation}\label{eq18}
d^{\,\prime}(f_{m_k}(|\gamma_k|), A)>\varepsilon, \quad k=1,2,\ldots
.
\end{equation}

\medskip
We cover the set $A$ with balls $B(x, \varepsilon/4),$ $x\in A.$
Since $A$ is compact, we may assume that $A\subset
\bigcup\limits_{i=1}^{M_0}B(x_i, \varepsilon/4),$ $x_i\in A,$
$i=1,2,\ldots, M_0,$ $1\leqslant M_0<\infty.$ By the definition,
$M_0$ depends only on $A,$ in particular, $M_0$ does not depend on
$k.$ Put
$$
\Gamma_k:=\Gamma(A_{m_k}, |\gamma_k|, D)\,.
$$
Let $\Gamma_{ki}:=\Gamma_{f_{m_k}}(x_i, \varepsilon/4,
\varepsilon/2),$ in other words, $\Gamma_{ki}$ consists of all paths
$\gamma:[0, 1]\rightarrow D$ such that $f_{m_k}(\gamma(0))\in S(x_i,
\varepsilon/4),$ $f_{m_k}(\gamma(1))\in S(x_i, \varepsilon/2)$ і
$\gamma(t)\in A(x_i, \varepsilon/4, \varepsilon/2)$ for $0<t<1.$ We
show that
\begin{equation}\label{eq6C}
\Gamma_k>\bigcup\limits_{i=1}^{M_0}\Gamma_{ki}\,.
\end{equation}
Indeed, let $\widetilde{\gamma}\in \Gamma_k,$ in other words,
$\widetilde{\gamma}:[0, 1]\rightarrow D,$ $\widetilde{\gamma}(0)\in
A_{m_k},$ $\widetilde{\gamma}(1)\in |\gamma_k|$ and
$\widetilde{\gamma}(t)\in D$ for $0\leqslant t\leqslant 1.$ Then
$\gamma^{\,*}:=f_{m_k}(\widetilde{\gamma})\in \Gamma(A,
f_{m_k}(|\gamma_k|), D^{\,\prime}).$ Since the balls $B(x_i,
\varepsilon/4),$ $1\leqslant i\leqslant M_0,$  form the coverage of
the compact set $A,$ we may find $i\in {\Bbb N}$ such that
$\gamma^{\,*}(0)\in B(x_i, \varepsilon/4)$ and $\gamma^{\,*}(1)\in
f_{m_k}(|\gamma_k|).$ By the relation~(\ref{eq18}),
$|\gamma^{\,*}|\cap B(x_i, \varepsilon/4)\ne\varnothing\ne
|\gamma^{\,*}|\cap (D^{\,\prime}\setminus B(x_i, \varepsilon/4)).$
Thus, by~\cite[Theorem~1.I.5.46]{Ku} there is $0<t_1<1$ such that
$\gamma^{\,*}(t_1)\in S(x_i, \varepsilon/4).$ We may assume that
$\gamma^{\,*}(t)\not\in B(x_i, \varepsilon/4)$ for $t>t_1.$ Set
$\gamma_1:=\gamma^{\,*}|_{[t_1, 1]}.$ By~(\ref{eq18}) it follows
that $|\gamma_1|\cap B(x_i, \varepsilon/2)\ne\varnothing\ne
|\gamma_1|\cap (D\setminus B(x_i, \varepsilon/2)).$ Thus,
by~\cite[Theorem~1.I.5.46]{Ku} there is $t_1<t_2<1$ such that
$\gamma^{\,*}(t_2)\in S(x_i, \varepsilon/2).$ We may assume that
$\gamma^{\,*}(t)\in B(x_i, \varepsilon/2)$ for any $t<t_2.$ Putting
$\gamma_2:=\gamma^{\,*}|_{[t_1, t_2]},$ we observe that a path
$\gamma_2$ is a subpath of $\gamma^{\,*},$ which belongs to
$\Gamma(S(x_i, \varepsilon/4), S(x_i, \varepsilon/2), A(x_i,
\varepsilon/4, \varepsilon/2)).$

Finally, $\widetilde{\gamma}$ has a subpath
$\widetilde{\gamma_2}:=\widetilde{\gamma}|_{[t_1, t_2]}$ such that
$f_{m_k}\circ\widetilde{\gamma_2}=\gamma_2,$ while
$$\gamma_2\in \Gamma(S(x_i, \varepsilon/4), S(x_i, \varepsilon/2),
A(x_i, \varepsilon/4, \varepsilon/2))\,.$$ Thus, the
relation~(\ref{eq6C}) is proved. Set
$$\eta(t)= \left\{
\begin{array}{rr}
4/\varepsilon, & t\in [\varepsilon/4, \varepsilon/2],\\
0,  &  t\not\in [\varepsilon/4, \varepsilon/2]\,.
\end{array}
\right. $$
Observe that $\eta$ satisfies the relation~(\ref{eqA2}) for
$r_1=\varepsilon/4$ and $r_2=\varepsilon/2.$ Since $f_{m_k}$
satisfies the relation~(\ref{eq2*A}), we obtain that
\begin{equation}\label{eq8C}
M_{\alpha}(\Gamma_{f_{m_k}}(x_i, \varepsilon/4,
\varepsilon/2))\leqslant
(4/\varepsilon)^{\alpha^{\,\prime}}\cdot\Vert Q\Vert_1<M_0<\infty\,.
\end{equation}
By~(\ref{eq6C}) and (\ref{eq8C}) and due to the subadditivity of the
modulus of families of paths, we obtain that
\begin{equation}\label{eq4B}
M_{\alpha}(\Gamma_k)\leqslant
\frac{4^{\alpha^{\,\prime}}M_0}{\varepsilon^{\alpha^{\,\prime}}}\int\limits_{D^{\,\prime}}Q(y)\,dm(y)\leqslant
M_1\cdot M_0<\infty\,.
\end{equation}
Arguing similarly to the proof of relations~(\ref{eq14A}) and using
the condition~(\ref{eq2B}), we obtain that
$M_{\alpha}(\Gamma_k)\rightarrow\infty$ as $k\rightarrow\infty,$
which contradicts with~(\ref{eq4B}). The resulting contradiction
proves the lemma. $\Box$
\end{proof}

\section{Equicontinuity of families of mappings fixing a point }

Given domains $D\subset X, D^{\,\prime}\subset X^{\,\prime},$ points
$a\in D,$ $b\in D^{\,\prime}$ and a number $M_0>0$ denote by ${\frak
S}_{a, b, M_0}(D, D^{\,\prime})$ the family of open discrete and
closed mappings $f$ of $D$ onto $D^{\,\prime}$ satisfying the
relation~(\ref{eq2*A}) for some $Q=Q_f,$ $\Vert
Q\Vert_{L^1(D^{\,\prime})}\leqslant M_0$ for any $y_0\in f(D),$ such
that $f(a)=b.$ The following statement was proved
in~\cite[Theorem~7.1]{SSD} in the case of a fixed function~$Q$ and
for the Euclidean space.

\medskip
\begin{theorem}\label{th5}
{\sl\, Let $D$ and $D^{\,\prime}$ be domains with finite Hausdorff
dimensions $\alpha$ and $\alpha^{\,\prime}\geqslant 2$ in spaces
$(X,d,\mu)$ and $(X^{\,\prime},d^{\,\prime}, \mu^{\,\prime}),$
respectively.  Assume that $X$ is locally connected, $\overline{D}$
is compact, $X^{\,\prime}$ is complete and supports
$\alpha^{\,\prime}$-Poincar\'{e} inequality, and that the measures
$\mu$ and $µ^{\,\prime}$ are doubling. Assume that $D$ has a weakly
flat boundary, none of the components of which degenerates into a
point, $\overline{D}$ is a compact set and $D^{\,\prime}$ be a
bounded regular domain in $X^{\,\prime},$ which is finitely
connected on the boundary. Then any $f\in {\frak S}_{a, b, M_0}(D,
D^{\,\prime})$ has a continuous extension
$\overline{f}:\overline{D}\rightarrow \overline{D^{\,\prime}}_P,$
while $\overline{f}(\overline{D})=\overline{D^{\,\prime}}_P$ and, in
addition, the family ${\frak S}_{a, b, M_0}(\overline{D},
\overline{D^{\,\prime}})$ of all extended mappings
$\overline{f}:\overline{D}\rightarrow \overline{D^{\,\prime}}_P$ is
equicontinuous in $\overline{D}.$
}
\end{theorem}

\medskip
\begin{proof} The possibility of continuous extension of
$f\in {\frak S}_{a, b, M_0}(D, D^{\,\prime})$ to a continuous
mapping $\overline{f}:\overline{D}\rightarrow
\overline{D^{\,\prime}}_P$ is a statement of Theorem~\ref{th1A}, as
well as the equality
$\overline{f}(\overline{D})=\overline{D^{\,\prime}}_P.$ The
equicontinuity of ${\frak S}_{a, b, M_0}(D, D^{\,\prime})$ at inner
points of $D$ follows by is a result of Theorem~\ref{th4A}.

\medskip
It remains to establish the equicontinuity of the family of extended
mappings $\overline{f}:\overline{D}\rightarrow
\overline{D^{\,\prime}}_P$ at the boundary points of the domain~$D.$

\medskip
We prove this statement from the opposite. Assume that the family
${\frak S}_{a, b, M_0}(\overline{D}, \overline{D^{\,\prime}})$ is
not equicontinuous at some point $x_0\in\partial D.$ Then there are
points $x_m\in D$ and mappings $f_m\in {\frak S}_{a, b,
M_0}(\overline{D}, \overline{D^{\,\prime}}),$ $m=1,2,\ldots ,$ such
that $x_m\rightarrow x_0$ as $m\rightarrow\infty,$ moreover,
\begin{equation}\label{eq15B}
m_P(f_m(x_m), f_m(x_0))\geqslant\varepsilon_0\,,\quad
m=1,2,\ldots\,.
\end{equation}
for some $\varepsilon_0>0,$ where $m_P$ is a metric in
$\overline{D^{\,\prime}}_P$ defined in Proposition~\ref{pr2}. We
choose in an arbitrary way the point $y_0\in D^{\,\prime},$ $y_0\ne
b,$ and join it to the point $b$ by some path in $D^{\,\prime},$
which we denote by $\alpha$ (this is possible due to
Lemma~\ref{lem5}). Let $A:=|\alpha|$ and let $A_m$ be a total
$f_m$-lifting of $\alpha$ starting at $a$ (it exists by
Lemma~\ref{lem9}). Observe that $d(A_m,
\partial D)>0$ due to the closeness of~$f_m.$ Now, the following two cases are possible:
either $d(A_m)\rightarrow 0$ as $m\rightarrow\infty,$ or
$d(A_{m_k})\geqslant\delta_0>0$ as $k\rightarrow\infty$ for some
increasing sequence of numbers $m_k$ and some $\delta_0>0.$

\medskip
In the first of these cases, obviously, $d(A_m, \partial D)\geqslant
\delta>0$ for some $\delta>0.$ Then, by Theorem~\ref{th2}, the
family $\{f_m\}_{m=1}^{\infty}$ is equicontinuous at the point
$x_0,$ however, this contradicts the condition~(\ref{eq15B}).

In the second case, if $d(f(A_{m_k}))\geqslant\delta_0>0$ for
sufficiently large $k,$ we also have that $d(A_{m_k}, \partial
D)\geqslant \delta_1>0$ for some $\delta_1> 0$ by Lemma~\ref{lem3}.
Again, by Theorem~\ref{th2}, the family $\{f_{m_k}\}_{k=1}^{\infty}
$ is equicontinuous at the point $x_0,$ and this contradicts the
condition~(\ref{eq15B}).

Thus, in both of the two possible cases, we came to a
contradiction~(\ref{eq15B}), and this indicates the incorrect
assumption of the absence of the equicontinuity of the family
${\frak S}_{a, b, M_0}(D, D^{\,\prime})$ in $\overline{D}.$ The
theorem is proved.~$\Box$
\end{proof}

\section{Examples}

To construct examples of mappings that satisfy the conditions of the
main results of the article, we will use the constructions related
to publications~\cite{Sev$_2$}, \cite{Sev$_3$} and~\cite{Skv}.

\medskip
\begin{example}\label{ex1}
Let $D$ be a half-disk $D:=\{z\in {\Bbb C}:z=x+iy, |z|<1, x>0\}$ on
the complex plane ${\Bbb C }.$ Put $f_1(z)=z^2.$ In this case,
$f_1(D)$ is the punctured disk $D^{\,\prime}:={\Bbb D}\setminus I,$
$I:=\{z\in {\Bbb C}: z=x+iy, y=0, x\in [0, 1)\},$ see
Figure~\ref{fig2}.

\begin{figure}[h]
\centering\includegraphics[width=400pt]{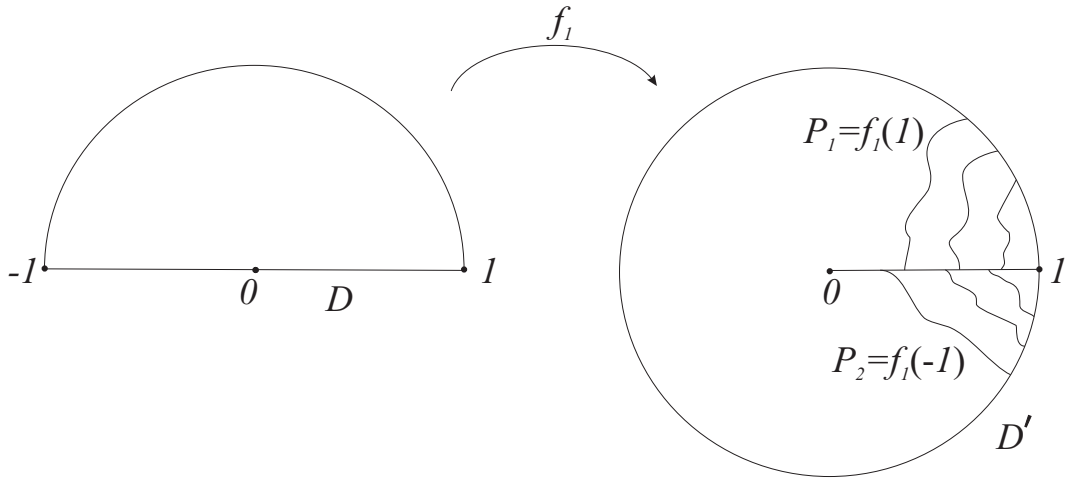} \caption{To
Example~\ref{ex1}}\label{fig2}
\end{figure}
Let $X=X^{\,\prime}={\Bbb R}^n,$ let $\mu$ and $\mu^{\,\prime}$ be
Lebesgue measures in ${\Bbb R}^n,$ and let $d(x, y)=d^{\,\prime}(x,
y)=|x-y|.$ Let us verify that for the domains $D$ and
$D^{\,\prime},$ the metric space $X=X^{\,\prime}={\Bbb R}^n,$
containing them, and also for the mapping $f_1,$ all conditions of
Theorem~\ref{th1A} are satisfied. Obviously, ${\Bbb R}^n$ is locally
connected and complete, and $\overline{D}$ is compact. Observe that
${\Bbb R}^n$ satisfies 1-Poincar\'{e} inequality (see
\cite[Theorem~10.5]{HK}) and, consequently, $n$-Poincar\'{e}
inequality. Obviously, the Lebesgue measure is doubling, in
addition, $D^{\,\prime}$ is finitely connected on the boundary.
Finally, the relation~(\ref{eq2*A}) holds with $Q\equiv 1$ because
$f_1$ is 1-quasiconformal mapping (see~\cite[Theorem~3.2]{MRV}).
Thus, all conditions of Theorem~\ref{th1A} are satisfied, so that
the mapping $f_1$ can be extended to a continuous mapping
$\overline{f_1}:\overline{D}\rightarrow \overline{D^{\,\prime}}_P.$
\end{example}

\medskip
In each subsequent example below, we retain the notation of the
previous examples, in particular, we retain the notation for the
domains $D$ and $D^{\,\prime}.$

\medskip
\begin{example}\label{ex2}
By slightly modifying Example~\ref{ex1}, we may construct a similar
mapping in the unit disk. For this purpose, we will map the unit
disk to the unit semidisk, guided by the sequential application of
the mappings $f_2(z)=-\frac{i(z-1)}{z+1}$ ($f_2:{\Bbb D}\rightarrow
{\Bbb H}^+,$ ${\Bbb H}^+:=\{z=x+iy\in {\Bbb C}, x>0\}$),
$f_3(z)=\sqrt{z}=\sqrt{r}e^{\varphi/2},$ $z=re^{i\varphi}$ ($f_3:
{\Bbb H}^+\rightarrow D_1,$ $D_1=\{z\in {\Bbb C}: z=re^{i\varphi},
0<\varphi<\pi/2\}$) and $f_4(z)=\frac{z+1}{-z+1}$
($f_4:D_1\rightarrow D$), we obtain a conformal mapping
$f_5:=f_2\circ f_3\circ f_4$ of the unit disk onto a domain $D.$
Now, the mapping $F_1:=f_1\circ f_5$ maps ${\Bbb D}$ onto
$D^{\,\prime}$ conformally, besides that, $F_1$ satisfies all the
conditions of Theorem~\ref{th1A}.
\end{example}

\medskip
\begin{example}\label{ex3}
It is fairly easy to provide a similar example of a mapping with
branching. For this purpose, we additionally set $f_6(z)=z^2,$ $z\in
{\Bbb D}.$ Put $F_2:=F_1\circ f_6.$ Now, $F_2$ is open discrete and
closed mapping of the unit disk onto a domain $D^{\,\prime}.$ Let
$\Gamma$ be a family of paths in ${\Bbb D}$ and let $M(\Gamma)$ be
the Euclidean modulus of family of paths $\Gamma$ in ${\Bbb C}.$
Now, by Theorem~\cite[Theorem~3.2]{MRV} we obtain that
$M(\Gamma)\leqslant 2\cdot M(f_6(\Gamma))=2\cdot M(F_2(\Gamma))$
because $F_2$ is a conformal mapping. Thus, the
relation~(\ref{eq2*A}) holds for $F_2$ with $Q\equiv 2.$ In
addition, by Theorem~\ref{th1A} the mapping $F_2$ can be extended to
a continuous mapping $\overline{F_2}:\overline{\Bbb D}\rightarrow
\overline{D^{\,\prime}}_P.$
\end{example}

\medskip
\begin{example}\label{ex4}
Let us indicate a similar example of a mapping with unbounded
characteristic. For this purpose,
fix a number $p\geqslant 1$ satisfying the condition $2/p<1.$ Put
$m\in {\Bbb N},$ $\alpha\in (0, 2/p).$ We define the sequence of
mappings~$f_m$ of the annulus $\{1<|z|<2\}$ onto the unit disk
${\Bbb D}$ as follows:
$$g_m(z):=\,\left
\{\begin{array}{rr} \frac{z}{|z|}(|z|-1)^{1/\alpha}\,, & 1+1/m^{\alpha}\leqslant|z|< 2, \\
\frac{(1/m)}{1+(1/m)^{\alpha}}\cdot z\,, & 0<|z|< 1+1/m^{\alpha} \ .
\end{array}\right.
$$
Note that $g_m$ satisfies  the condition~(\ref{eq2*A}) for
$Q=\frac{1+|z|^{\,\alpha}}{\alpha |z|^{\,\alpha}}$ at any $z_0\in
D^{\,\prime},$ moreover, $Q\in L^p({\Bbb D})$ (see, for example, the
reasonings obtained under the consideration
of~\cite[Proposition~6.3]{MRSY}). Putting $f_8(z)=2z,$ we observe
that the mappings
\begin{equation}\label{eq3A}
h_m:=g_m\circ f_8\circ F_2
\end{equation}
transform ${\Bbb D}$ onto $D^{\,\prime},$ in addition, $h_m$
satisfies~(\ref{eq2*A}) for $Q=2\frac{1+|z|^{\,\alpha}}{\alpha
|z|^{\,\alpha}}.$ By Theorem~\ref{th1A} the mapping $h_m$ can be
extended to a continuous mapping $\overline{h_m}:\overline{\Bbb
D}\rightarrow \overline{D^{\,\prime}}_P$ for any $m=1,2,\ldots .$

\medskip
Note that for the constructed sequence of mappings $h_m,$ $m=1,2
\ldots, $ all conditions and the conclusion of Theorem~\ref{th2} are
also satisfied. For this purpose, first of all, we note that there
are infinitely many continua $A$ satisfying condition
$d(h_m^{\,-1}(A),
\partial {\Bbb D})\geqslant~\delta,$ $m=1,2,\ldots, $ with some
$\delta> 0,$ since all the mappings $g_m$ for $m>2^{1/\alpha}$ and
$3/2<|z|<2$ are a fixed mapping by $m,$ equals to
$\frac{z}{|z|}(|z|-1)^{1/\alpha},$ besides that, the mapping
$f_8\circ F_2$ does not depend on $m.$ Observe that ${\Bbb D}$ is a
weakly flat at inner and boundary points (see~\cite[Theorem~10.12,
Theorem~17.10]{Va}, cf.~\cite[Lemma~2.2]{SevSkv$_2$}). Observe that
$D^{\,\prime}$ is regular by Riemann's mapping theorem. Thus, all
the conditions of Theorem~\ref{th2} are fulfilled, so that the
family of mappings $\overline{h_m}:\overline{\Bbb D}\rightarrow
\overline{D^{\,\prime}}_P$ is equicontinuous in $\overline{D}.$

It can be shown that the sequence of inverse mappings $h_m^{\,-1}$
have continuous extension to  $\overline{D^{\,\prime}}_P,$ but is
not an equicontinuous family in $\overline{D^{\,\prime}}_P$ (for
this purpose it is necessary to construct a prime end with an
impression at the point $0,$ and then apply similar arguments using
in the consideration of \cite[Example~1]{SevSkv$_2$}).
\end{example}

\medskip
\begin{example}\label{ex5}
Finally, consider a similar example of mappings in a metric space.
For this purpose, we essentially use our construction
from~\cite{Sev$_4$}, cf.~\cite{Skv}. Regarding this example, we will
need some definitions related to the factor space, discontinuously
acting groups of mappings, and the normal neighborhood of a point.
For these definitions, we also refer the reader to~\cite{Sev$_4$}.

\medskip
Let $G$ and $G^{\,*}$ be groups of M\"{o}bius transformations of the
unit disk ${\Bbb D}$ onto itself, acting discontinuously and not
having fixed points in ${\Bbb D}.$ Suppose also that ${\Bbb D}/G$
and  ${\Bbb D}/G^{\,*}$  are complete $2$-Ahlfors regular spaces
with $2$-Poincar\'{e} inequality. Let $\pi:{\Bbb D}\rightarrow {\Bbb
D}/G$ and $\pi_*:{\Bbb D}\rightarrow {\Bbb D}/G^{\,*}$ be the
natural projections of ${\Bbb D}$ onto ${\Bbb D}/G$ and ${\Bbb
D}/G^{\,*},$ respectively, and let $p_0\in {\Bbb D}/G$ and
$p^{\,*}_0\in {\Bbb D}/G^{\,*}$ be such that $\pi(0)=p_0$ and
$\pi_*(0)=p^{\,*}_0.$ In what follows, $h$ denotes the hyperbolic
metric in ${\Bbb D},$ $\widetilde{h}$ and $\widetilde{h}_*$ denote
the metric in ${\Bbb D}/G$ and ${\Bbb D}/G^{\,*},$ respectively,
$dv$ denotes the element of hyperbolic area in ${\Bbb D},$ and
$d\widetilde{v}$ and $d\widetilde{v}_*$ denote the elements of the
area in ${\Bbb D}/G$ and ${\Bbb D}/G^{\,*}$ (see~\cite{Sev$_4$}).
Namely,
\begin{equation}\label{eq2BC}
dv(x)=\frac{4\, dm(x)}{{(1-|x|^2)}^4}\,,
\end{equation}
\begin{equation}\label{eq3C}
h(x, y)=\log\,\frac{1+t}{1-t}\,,\quad
t=\frac{|x-y|}{\sqrt{|x-y|^2+(1-|x|^2)(1-|y|^2)}}\,,
\end{equation}
\begin{equation}\label{eq2C}
\widetilde{h}(p_1, p_2):=\inf\limits_{g_1, g_2\in G}h(g_1(z_1),
g_2(z_2))\,,
\end{equation}
\begin{equation}\label{eq5C}
\widetilde{v}(A):=v(P\cap\pi^{\,-1}(A))\,,
\end{equation}
where
%
$P=\{x\in {\Bbb D}: h(x, x_0)<h(x, T(x_0))\quad {\rm
for\,\,all}\quad T\in G\setminus\{I\}\}.$
Similarly we may define $\widetilde{h}_*$ and $\widetilde{v}_*$ in
${\Bbb D}/G^{\,*}.$ Observe that the length element $ds_h$ in the
metric space $({\Bbb D}, h)$ equals to $\frac{2|dz|}{1-|z|^2}.$

\medskip
Let $r_1>0$ be the radius of a disk $\widetilde{B}(p_0, r_1)=\{p\in
{\Bbb D}/G: \widetilde{h}(p, p_0)<r_1\}$ centered at a point~$p_0,$
lying in some normal neighborhood $U$ of $p_0$ with its closure, and
let $R_1>0$ be the radius of a disk $\widetilde{B}(p^{\,*}_0,
R_1)=\{p_*\in {\Bbb D}/G^{\,*}: \widetilde{h}_*(p_*,
p^{\,*}_0)<R_1\}$ centered at a point~$p^{\,*}_0,$ entirely (with
its closure) lying in some normal neighborhood $U_{\,*}$ of
$p^{\,*}_0.$

\medskip
Then by the definition of the natural projection $\pi,$ as well as
the definition of the hyperbolic metric~$h$ and the
metrics~$\widetilde{h}$ in~(\ref{eq2C}) we have, that $\pi(B(0,
r_0))=\widetilde{B}(p_0, r_1)$ and $\pi_*(B(0,
R_0))=\widetilde{B}(p_0, R_1),$ where $r_0:=(e^{r_1}-1)/(e^{r_1}+1)$
and $R_0:=(e^{R_1}-1)/(e^{R_1}+1).$ Let $V$ be a neighborhood of the
origin containing $B(0, r_0)$ such that $\pi$ maps $V$ onto $U$
homeomorphically, and let $V_*$ be a neighborhood of the origin
containing $B(0, R_0)$ such that $\pi_*$ maps $V_*$ onto $U_*$
homeomorphically.

\medskip
Put $f_9(z)=R_0z$ and $f_{10}(z)=\frac{z}{r_0}.$ In this case, the
family of mappings
$$\widetilde{H}_m(z)=(f_9\circ h_m\circ f_{10})(z)$$
map the ball $B(0, r_0)$ onto $B(0, R_0)\setminus J_0,$
\begin{equation}\label{eq5F}
J_0:=\{z\in {\Bbb C}: z=x+iy, x\in[0, R_0)\}\,.
\end{equation}
Recall that, by~(\ref{eq3A})
\begin{equation}\label{eq4A}
\widetilde{H}_m=f_9\circ g_m\circ f_8\circ F_1\circ f_6\circ
f_{10}\,,
\end{equation}
where
$$f_6(z)=z^2\,,\quad f_8(z)=2z\,,\quad
f_9(z)=R_0z\,, \quad f_{10}(z)=\frac{z}{r_0}\,,$$
$$g_m(z):=\,\left
\{\begin{array}{rr} \frac{z}{|z|}(|z|-1)^{1/\alpha}\,, & 1+1/m^{\alpha}\leqslant|z|< 2, \\
\frac{(1/m)}{1+(1/m)^{\alpha}}\cdot z\,, & 0<|z|< 1+1/m^{\alpha} \ .
\end{array}\right.$$
and $F_1(z)$ is some conformal mapping. Put
$$H_m=\pi_*\circ \widetilde{H}_m\circ (\pi|_V)^{\,-1}\,.$$
Observe that the $H_m$ map $\pi(B(0, r_0))\subset {\Bbb D}/G$ onto
$\pi^{\,*}(B(0, R_0))\setminus J_0\subset {\Bbb D}/G^{\,*},$ where
$J_0$ is defined in~(\ref{eq5F}), see Figure~\ref{fig7}.
\begin{figure}[h]
\center{\includegraphics[scale=0.4]{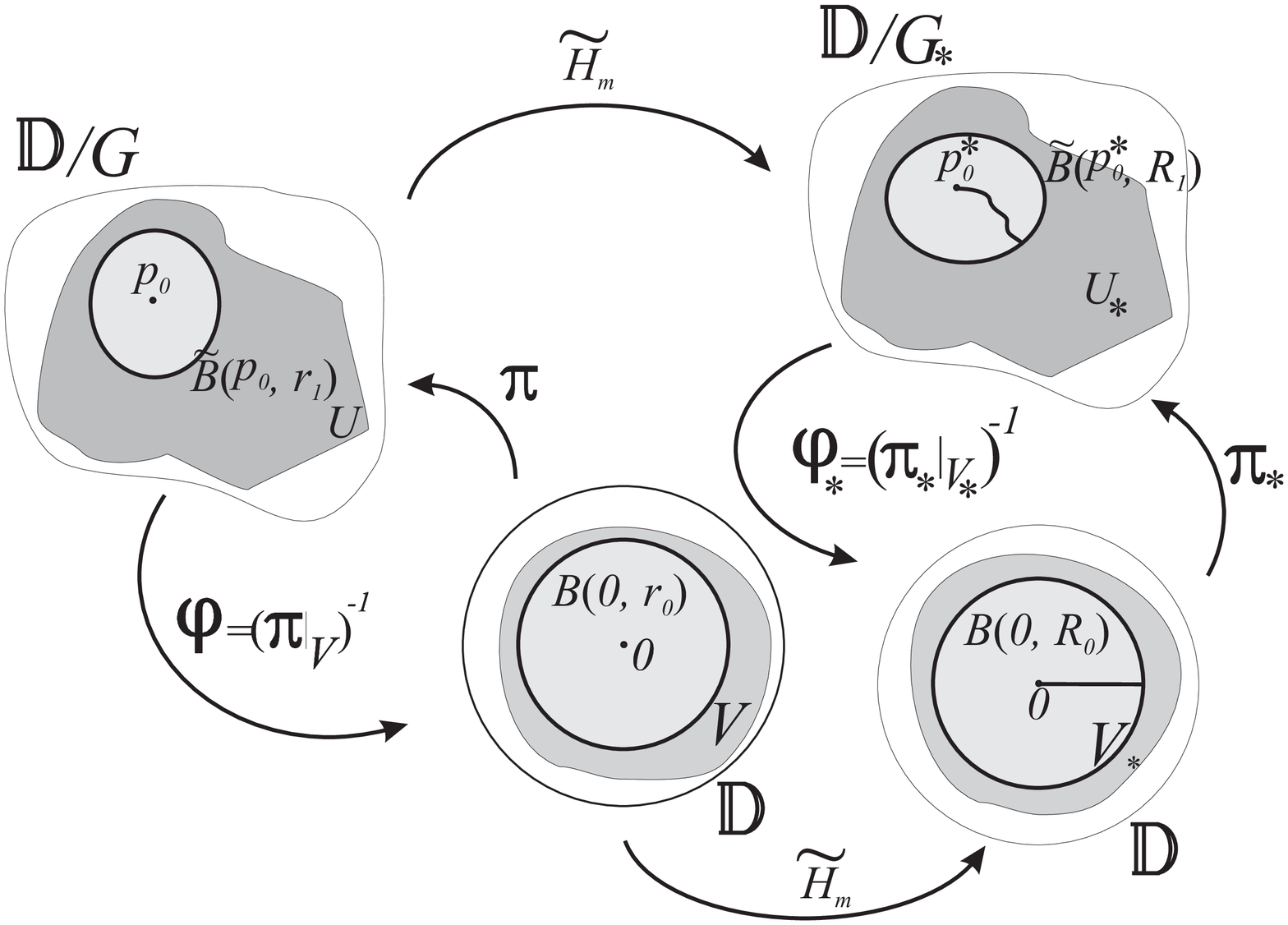}}
\caption{Example~\ref{ex5}}\label{fig7}
\end{figure}
Our immediate goal is to establish an estimate for the distortion of
the modulus of families of paths~(\ref{eq2*A}) for the mappings
$H_m,$ $m=1,2,\ldots .$ We will do this, focusing on the scheme of
consideration of~\cite[Example~4.4]{Skv}.

\medskip
In what follows, $M_h(\Gamma)$ denotes the hyperbolic modulus of
families of paths, i.e., the modulus of a path family $\Gamma$ in
${\Bbb D}$ defined by~(\ref{eq13.5}) for $p=2,$ where $d\mu$ denotes
the element of hyperbolic area $dv$ in~(\ref{eq2BC}), and the
admissibility of the function $\rho$ in~(\ref{eq13.2}) must be
understood with respect to the element of length $ds_h$
corresponding to~(\ref{eq3C}). It is known that
$M_h(\Gamma)=M(\Gamma),$ where $M(\Gamma)$ denotes the Euclidean
modulus of families of paths (see \cite[Remark~5.2]{Sev$_1$}).
Similarly, let $\widetilde{M}(\Gamma)$ denotes the modulus of family
$\Gamma$ in ${\Bbb D}/G$ (or ${\Bbb D}/G^{\,*}$). In this case, we
mean the modulus in~(\ref{eq13.5}) for $p=2,$ where $d\mu$ denotes
the element of hyperbolic area $d\widetilde{v}$ in~(\ref{eq5C}), and
the admissibility of the function $\rho$ in~(\ref{eq13.2}) must be
understood with respect to the element of length
$ds_{\widetilde{h}}$ corresponding to~(\ref{eq2C}).

\medskip
Let $\Gamma$ be a family of paths in $\widetilde{B}(p_0, r_1).$ Then
by the definition of the normal neighborhood $U$ we obtain that
$\widetilde{M}(\Gamma)=M_h(\varphi(\Gamma)),$ where
$\varphi:=(\pi|_V)^{\,-1}.$ By~\cite[Remark~5.2]{Sev$_1$},
$\widetilde{M}(\Gamma)=M(\varphi(\Gamma)).$ Further, since $f_{10}$
is conformal, $\widetilde{M}(\Gamma)=M((f_{10}\circ
\varphi)(\Gamma)).$ By~\cite[Theorem~3.2]{MRV},
$\widetilde{M}(\Gamma)\leqslant 2 M((f_6\circ f_{10}\circ
\varphi)(\Gamma)),$ where $f_6(z)=z^2.$ Since $f_8\circ F_1$ is a
conformal mapping, we also have that $\widetilde{M}(\Gamma)\leqslant
2 M((f_8\circ F_1\circ f_6\circ f_{10}\circ \varphi)(\Gamma)).$ Now,
by Example~\ref{ex4}
\begin{equation}\label{eq6D}
\widetilde{M}(\Gamma)\leqslant 2 M((f_8\circ F_1\circ f_6\circ
f_{10}\circ \varphi)(\Gamma))\leqslant\int\limits_{B(0,
R_0)}Q_*(y)\rho_*^2(y)\,dm(y)
\end{equation}
for any $\rho_*\in {\rm\,adm}(\widetilde{H}_m\circ
(\pi|_V)^{\,-1})(\Gamma),$ where
$Q_*(y)=2\frac{1+|\frac{y}{R_0}|^{\,\alpha}}{\alpha
|\frac{y}{R_0}|^{\,\alpha}}.$ Let $\rho_*\in
{\rm\,adm}_h(\widetilde{H}_m\circ (\pi|_V)^{\,-1})(\Gamma),$ that
is,
$\int\limits_{\gamma}\rho(z)\,ds_h(z)=\int\limits_{\gamma}\frac{2\rho_*(z)}{
(1-|z|^2)}|dz|\geqslant 1$ for any (locally rectifiable) path
$\gamma\in (\widetilde{H}_m\circ (\pi|_V)^{\,-1})(\Gamma).$ Then
by~(\ref{eq6D}) and~(\ref{eq2BC}) as well as by the definition of a
normal neighborhood we obtain that
\begin{equation}\label{eq7D}
\widetilde{M}(\Gamma)\leqslant 4\int\limits_{B(0,
R_0)}Q_*(y)\frac{\rho_*^2(y)}{(1-|y|^2)^2}\,dm(y)=\int\limits_{B(0,
R_0)}Q_*(y)\rho_*^2(y)\,dv(y)\,.
\end{equation}
Let $\rho_*\in {\rm\,adm} H_m(\Gamma),$ then by the definition of a
normal neighborhood $U_*$ we obtain that $\rho_*\in
{\rm\,adm}(\widetilde{H}_m\circ (\pi|_V)^{\,-1})(\Gamma).$ Now, it
follows from~(\ref{eq7D}) that
\begin{equation}\label{eq7E}
\widetilde{M}(\Gamma)\leqslant \int\limits_{\widetilde{B}(p_0^*,
R_1)}Q_1(p_*)\rho_*^2(p_*)\,d\widetilde{v}_*(p_*)\,,
\end{equation}
where
$Q_1(p_*):=2\frac{1+{|\frac{\varphi_*(p_*)}{R_0}|}^{\,\alpha}}{\alpha
{|\frac{\varphi_*(p_*)}{R_0}|}^{\,\alpha}}.$
Observe that the function $Q_1(p_*)$ is integrable in
$\widetilde{B}(p_0^*, R_1)$ because
$$\int\limits_{\widetilde{B}(p_0^*, R_1)}Q(p_*)\,d\widetilde{v}_*(p_*)
=\int\limits_{B(0,
R_0)}\frac{4(1+|\frac{z}{R_0}|^{\,\alpha})\,dm(z)}{(1-|z|^2)^2\alpha
|\frac{z}{R_0}|^{\,\alpha}}\leqslant$$$$\leqslant C\cdot
\int\limits_{\Bbb D}\frac{(1+|z|^{\,\alpha})\,dm(z)}{\alpha
|z|^{\,\alpha}}<\infty\,.$$
Thus, all the mappings $H_m,$ $m=1,2\ldots $ satisfy the
condition~(\ref{eq7E}) and, consequently, the relation~(\ref{eq2*A})
with some general integrable function $Q_1=Q_1(p_*).$

\medskip
Let us check that all conditions of Theorems~\ref{th1A}
and~\ref{th2} are satisfied for the family of mappings $H_m,$
$m=1,2,\ldots .$ Observe that ${\Bbb D}/G$ is locally connected (see
e.g.~\cite[Proposition~1.1]{Sev$_4$},
cf.~\cite[Proposition~3.14]{Ap}). By the construction, $H_m$ act
between domains $\widetilde{B}(p_0, r_1)$ and
$\widetilde{B}(p^{\,*}_0, R_1)\setminus \pi_*(J_0)$ with compact
closures, in addition, all $H_m$ are discrete, open and closed. By
the assumption, the spaces ${\Bbb D}/G$ and ${\Bbb D}/G^{\,*}$ are
complete $2$-Ahlfors regular spaces with $2$-Poincar\'{e}
inequality. In particular, the measures $\mu:=\widetilde{v}$ and
$\mu^{\,\prime}:=\widetilde{v}_*$ are doubling. Obviously,
$\widetilde{B}(p^{\,*}_0, R_1)\setminus \pi_*(J_0)$ is bounded and,
in addition, finitely connected on the boundary because $B(0,
R_0)\setminus J_0$ is that. All of the conditions of
Theorem~\ref{th1A} are fulfilled, thus, any $H_m,$ $m=1,2,\ldots ,$
has a continuous extension
$\overline{H_m}:\overline{\widetilde{B}(p_0,
r_1)}\rightarrow\overline{\widetilde{B}(p^{\,*}_0, R_1)\setminus
\pi_*(J_0)}_P$ such that
$\overline{H_m}(\overline{\widetilde{B}(p_0,
r_1)})=\overline{\widetilde{B}(p^{\,*}_0, R_1)\setminus
\pi_*(J_0)}_P.$

\medskip
Note that for the constructed sequence of mappings $H_m,$ $m=1,2
\ldots, $ satisfies and the conclusion of Theorem~\ref{th2} , as
well. Indeed, observe that there are infinitely many continua $A$
satisfying condition $d(H_m^{\,-1}(A),
\partial {\Bbb D})\geqslant~\delta,$ $m=1,2,\ldots, $ with some
$\delta> 0,$ since all the mappings $g_m$ in~(\ref{eq4A}) are a
fixed mapping by $m$ for $m>2^{1/\alpha}$ and $3/2<|z|<2,$ and the
remaining mappings that make up $H_m,$ are fixed and do not depend
on the index $m=1,2,\ldots. $ Observe that $\widetilde{B}(p_0, r_1)$
is a weakly flat at inner and boundary points, because
$\widetilde{B}(p_0, r_1)$ is hyperbolic isometric to the disk $B(0,
r_0)$ and, in addition, $B(0, r_0)$ is weakly flat with the respect
to hyperbolic metric and hyperbolic measure because the hyperbolic
and Euclidean metrics are equivalent on compact sets in ${\Bbb D},$
and $B(0, r_0)$ is weakly flat with the respect to Euclidean metric
and Lebesgue measure by~\cite[Theorem~10.12, Theorem~17.10]{Va},
cf.~\cite[Lemma~2.2]{SevSkv$_2$}. Observe that
$\widetilde{B}(p^{\,*}_0, R_1)\setminus \pi_*(J_0)$ is regular by
Riemann's mapping theorem. Thus, all the conditions of
Theorem~\ref{th2} are fulfilled, so that the family of mappings
$\overline{H_m}:\overline{\widetilde{B}(p_0,
r_1)}\rightarrow\overline{\widetilde{B}(p^{\,*}_0, R_1)\setminus
\pi_*(J_0)}_P$ is equicontinuous in $\overline{\widetilde{B}(p_0,
r_1)}.$
\end{example}

\medskip
{\bf \noindent Evgeny Sevost'yanov} \\
{\bf 1.} Zhytomyr Ivan Franko State University,  \\
40 Bol'shaya Berdichevskaya Str., 10 008  Zhytomyr, UKRAINE \\
{\bf 2.} Institute of Applied Mathematics and Mechanics\\
of NAS of Ukraine, \\
1 Dobrovol'skogo Str., 84 100 Slavyansk,  UKRAINE\\
esevostyanov2009@gmail.com

\end{document}